\def\obs#1{{\bf (*** #1 ***)} }
\def\obs#1{}     % Remova esta linha para rodar a versao 1
\documentclass[twoside,letterpaper,draft,11pt]{amsart}
\usepackage{amsfonts,amssymb,amsmath,amsthm,latexsym,xspace,amscd,amssymb,mathtools,color,enumerate}
\usepackage[mathscr]{eucal}
\usepackage[all]{xy}
\usepackage[english]{babel}
\usepackage[x11names]{xcolor}

\makeatletter\renewcommand\theenumi{\@roman\c@enumi}\makeatother
\theoremstyle{thm}
\newtheorem{thm}{Theorem}[section]
\newtheorem{cor}[thm]{Corollary}
\theoremstyle{lem}
\newtheorem{lem}[thm]{Lemma}
\newtheorem{prop}[thm]{Proposition}
\newtheorem{defn}[thm]{Definition}
\theoremstyle{rem}
\newtheorem{rem}[thm]{Remark}
\newtheorem{exe}[thm]{Example}

\newcommand{\Hh}{\mathcal{H}}

\newcommand{\I}{\mathcal{I}}

\newcommand{\af}{\alpha}

\newcommand{\bt}{\beta}
\newcommand{\cl}{\rm cl}

\newcommand{\om}{\omega}
\newcommand{\ot}{\otimes}

\newcommand{\m}{^{-1}}
\newcommand{\id}{\operatorname{id}}

\begin{document}
	
\thispagestyle{empty}

\title[On partial Galois abelian extensions]{On partial Galois abelian extensions}

\author[A. Ca\~nas]{Andr\'es Ca\~nas}
\address{Escuela de Matematicas, Universidad Industrial de Santander, Cra. 27 Calle 9  UIS
	Edificio 45, Bucaramanga, Colombia}\email{ @uis.edu.co }

\author[V. Mar\'{\i}n]{V\'{\i}ctor Mar\'{\i}n}
\address{Departamento de Matem\'{a}ticas y Estad\'{i}stica, Universidad del Tolima, Santa Helena\\
	Ibagu\'{e}, Colombia} \email{vemarinc@ut.edu.co}

%
%\author[A.  Paques]{Antonio Paques}
%\address{Instituto de
%Matem\'atica  Universidade  Federal do Rio Grande do Sul,  Avenida Bento Gon\c calves 9500,
%91509-900 Porto Alegre, RS, Brasil.}
%\email{paques@mat.ufrgs.br}

\author[H. Pinedo]{H\'{e}ctor Pinedo}
\address{Escuela de Matematicas, Universidad Industrial de Santander, Cra. 27 Calle 9  UIS
	Edificio 45, Bucaramanga, Colombia}\email{ hpinedot@uis.edu.co}

\thanks{{\bf  Mathematics Subject Classification}: Primary . Secondary .}
\date{\today}

\begin{abstract}
	In this article we %continue  the Galois theory of commutative rings under partial actions of finite groups developed by Dokuchaev, Ferrero and Paques in \cite{DFP}, and we 
construct the inverse semigroup of equivalence classes of partial Galois abelian extensions of a commutative ring  $R$ with same group $G$, called the Harrison partial inverse semigroup.
\end{abstract}
\maketitle
\setcounter{tocdepth}{1}

\section{Introduction}

In the 1960's, M. Auslander and O. Goldman introduced in \cite{AG} the notion of Galois extension for commutative rings. After that, S. U. Chase, D. K. Harrison and A. Rosenberg developed in \cite{CHR} the Galois theory for commutative rings extending the classical theory over fields. One of the main results of \cite{CHR} is Theorem 2.3 which has two parts. Let $R\subset S$ be a Galois extension of commutative rings with Galois group $G$. 
The first part established a bijective correspondence between subgroups of $G$ and $R$-subrings of $S$ which are $G$-strong. In the second part, it was shown that if $H$ is a normal subgroup of $G$ then the subring $S^H$ of $S$ of invariantes by the action of $H$ is a Galois extension of $R$ with Galois group $G/H$. \smallbreak

A Galois extension of commutative rings $R\subset S$ with Galois group $G$ is called {\it abelian} when $G$ is an abelian group. For a fixed abelian group $G$ and a fixed commutative ring $R$,  D. K. Harrison constructed in \cite{H} a group $\Hh(G,R)$, which is called the {\it Harrison group}. The elements of $\Hh(G,R)$  are the classes of $G$-isomorphism of abelian extensions of $R$ with group $G$. To define a binary operation in $\Hh(G,R)$, D. K. Harrison used the second part of Theorem 2.3 of \cite{CHR}; see details in \cite{H}. \smallbreak

M. Dokuchaev, M. Ferrero and A. Paques developed in \cite{DFP} the Galois theory for commutative rings when the group acts partially. The results proved in \cite{DFP} generalizes many results of \cite{CHR}. For instance, Theorem 5.1 of \cite{DFP} generalizes the first part of Theorem 2.3 of \cite{CHR}. However, there is no generalization of the second part of Theorem 2.3 of \cite{CHR} to the context of partial actions. \smallbreak

Given a finite abelian group $G$ and a commutative ring $R$, we consider the set  $\Hh_{par}(G,R)$ of equivalence classes of partial Galois extensions of $R$ with group $G$. In particular, we have $\Hh(G,R)\subset \Hh_{par}(G,R)$. The main purpose of this paper is to provide a structure of inverse semigroup for $\Hh_{par}(G,R)$. For such, we will generalize the second part of Theorem 2.3 of \cite{CHR} to the context of partial actions and we will follow similar ideas to those used in \cite{H}.
\smallbreak

The paper is organized as follows. The basic notions and results that we used throughout the paper are presented in Section \ref{notions}. In Section \ref{quo}, we prove the second part of Theorem 2.3 of \cite{CHR} for partial actions. Following ideas from \cite{DFP}, we study in Section \ref{enve}  (partial) Galois extensions $G$-isomorphic. 
Finally, in Section 5,  we prove that $\Hh_{par}(G,R)$ is an inverse semigroup and we illustrate the binary operation  of $\Hh_{par}(G,R)$ with a concrete example.

\subsection*{Conventions}\label{subsec:conv}
Throughout the paper, rings are always considered commutative with identity element. Each ring homomorphism is unitary, that is, it sends identity element in identity element. The extensions of rings have same identity element. If $S$ and $T$ are extensions of a same ring $R$ then $S\otimes T$ means $S\otimes_R T$.
Moreover, $k$ will denote an associative and commutative ring with unity and $G$ a group. The identity element of $G$ will be denoted by $1$.

\section{Preliminaries}\label{notions}

In this section we present the background about partial actions and globalization that will be used in the paper.

\subsection{Partial action of groups} A \textit{partial action} of a group $G$ on a $k$-algebra $S$ is a pair $\alpha=(S_g,\alpha_g)_{g\in G}$
that satisfies:  
\begin{enumerate}[\hspace{.35cm}(P1)]
	    \item for each $g\in G$, $S_g$ is an ideal of $S$ and $\af_{g}:S_{g\m}\to S_g$ is a $k$-algebra isomorphism,\vspace{.08cm}
		\item $S_1=S$ and $\alpha_1=\id_S$, \vspace{.08cm}
		\item $\alpha_g(S_{g^{-1}}\cap S_h)=S_g\cap S_{gh}$, for all $g,h \in G$,\vspace{.08cm}
		\item $\alpha_g \circ \alpha_h(x)=\alpha_{gh}(x)$, for all  $ x\in S_{h^{-1}}\cap S_{(gh)\m}$ and $g,h \in G$.
\end{enumerate}
The partial action $\af$ is called {\it unital} if every ideal $S_g$ is unital, that is, there exists $1_g\in S$ such that $S_g=S1_g$.
Notice that the conditions (P3) and (P4) imply that $\alpha_{gh}$  is an extension of $\alpha_g \circ \alpha_h$, for every $g,h\in G$.\smallbreak

A partial action $\af$ is said {\it global} if $S_g=S$, for all $g\in G$. Global actions of a group $G$ on a $k$-algebra $T$ induce, by restriction, partial actions on any ideal $S$ of $T$. Indeed, given a global action $\bt=(T_g,\beta_g)_{g\in G}$  of $G$ on $T$ we consider the ideals $S_g=S\cap \beta_g(S)$ of $S$ and the $k$-algebra isomorphisms $\af_g=\beta_g|_{S_{g\m}}$. Then $\af=(S_g,\af_g)$ is a partial action of $G$ on $S$. Partial actions obtained in this way are called {\it globalizable}. The precise definition is given below. \smallbreak   

Let $\alpha=(S_g,\alpha_g)_{g\in G}$ be a partial action of a group $G$ on a $k$-algebra $S$. A {\it globalization of $\af$}  is a global action $\bt=(T_g,\beta_g)_{g\in G}$  of $G$ on a $k$-algebra $T$ that satisfies:
\begin{enumerate}[\hspace{.35cm}(G1)]
	\item $S$ is an ideal of $T$, \vspace{.08cm}
	\item $S_g=S \cap \beta_g(S)$, for all $g\in G$,\vspace{.08cm}
	\item $\beta_g (x)=\alpha_g(x)$, for all $x \in S_{g^{-1}}$,\vspace{.08cm}
	\item $T=\sum_{g \in G}\beta_g(S)$.
\end{enumerate}
If $\af$ admits a globalization we say that $\af$ is {\it globalizable}. The globalization of $\af$ is
unique, up to isomorphism, and will be denote by $(T,\beta)$. Also, Theorem 4.5 of \cite{DE} implies that  unital partial actions admit globalizations. More details related to globalization can be seen in \cite{DE}.

From now on we assume that  $G$ is a finite group, $\af=(S_g,\alpha_g)_{g\in G}$ is a unital partial action of $G$ on a $k$-algebra $S$ such that $S_g=S1_g$, for all $g\in G$, and $\beta=(T_g,\beta_g)_{g\in G}$ is a global action of $G$ on a $k$-algebra $T$ which is the globalization of $\af$. Notice that $1_S$ is a central idempotent element of $T$ and $S=T1_S$. Moreover, it was proved in \cite[p.79]{DFP} that 
\begin{align}\label{form1:partial-action}
&1_g=\bt_g(1_S)1_S,\quad \alpha_g(s1_{g\m})=\beta_g(s)1_S,\quad\af_g(\af_h (s1_{h\m})1_{g\m})=\af_{gh}(s1_{(gh)\m})1_g,&  
\end{align} 
for all $g,h\in G$ and $s\in S$. Let $H=\{h_1=1, h_2, \dots, h_m\}$ be a subgroup  of $G$. It was defined in \cite[p.79]{DFP} the map $\psi_H:T\to T$  by 
\begin{equation} \label{fih}
\psi_H(t)=\sum_{1\leq l \leq m}\sum_{i_1< \cdots < i_l}(-1)^{l+1}\beta_{h_{i_1}}(1_S)\beta_{h_{i_2}}(1_S)\cdots \beta_{h_{i_l}}(1_S)\beta_{h_{i_l}}(t),\quad t\in T.
\end{equation}
This map can be rewritten as
\begin{equation*}
\psi_H(t)=\sum\limits_{i=1}^m\beta_{h_i}(t)e_i,\quad\text{for all}\,\, t\in T,
\end{equation*}
where $e_i$'s are the following idempotents of $T$: 
\begin{equation}\label{losei}
e_1=1_S, \quad e_i=(1_T-1_S)(1_T-\beta_{h_2}(1_S))\dots (1_T-\beta_{h_{i-1}}(1_S))\beta_{h_i}(1_S),\quad 2\leq i\leq m.
\end{equation}

Since $1_S$ is a central idempotent of $T$, given $2\leq i\leq m$ we have
\begin{align*}\bt_g(e_i)1_S&=(1_T-\bt_g(1_S))(1_T-\beta_{gh_2}(1_S))\dots (1_T-\beta_{gh_{i-1}}(1_S))\beta_{gh_i}(1_S)1_S\\
&=(1_S-\bt_g(1_S)1_S)(1_S-\beta_{gh_2}(1_S)1_S)\dots (1_S-\beta_{gh_{i-1}}(1_S)1_S)\beta_{gh_i}(1_S)1_S\\
&\stackrel{\eqref{form1:partial-action}}=(1_S-1_g)(1_S-1_{gh_2})\dots (1_S-1_{gh_{i-1}})1_{gh_i}.
\end{align*}
Thus 
\begin{equation}\label{btgei}
\bt_g(e_1)1_S=1_g\,\,\, \,\text{and}\,\,\,\, \bt_g(e_i)1_S=\prod\limits_{j=2}^i(1_S-1_{gh_{i-1}})1_{gh_i}, \quad 2\leq i\leq m.
\end{equation}

We recall from \cite{DFP} that $S^\alpha:=\{s\in S :\alpha_g(s1_{g^{-1}})=s1_g,\,\text{for all}\,\, g\in G\}$ is called the {\it subalgebra of invariants} of $S$ under $\af$. 
If $\alpha$ is global then $S^\alpha$ is the classical subalgebra of invariants, i~e. $S^\alpha=S^{G}=\{x\in S: \af_g(s)=s,\,\text{for all}\,\, g\in G\}$. \smallbreak

%If $\alpha_H$ is the partial action of $H$ on $S$ given by restriction then $S^\alpha\subseteq S^{\alpha_H}$. \smallbreak

For a subgrupo $H$ of $G$, we shall denote by $\af_H$ the partial action of $H$ on $S$ obtained by restriction of $\af$, i.~e. $\af_H=(S_h,\af_h)_{h\in H}$. Some properties of the map $\psi_H$ are given in the next.

\begin{prop}\label{pro3.1}
Let $H=\{h_1=1, h_2, \dots, h_m\}$ be a subgroup  of $G$ and  $\psi_H:T\to T$ the map defined in \eqref{fih}. Then:
	\begin{enumerate}[\rm (i)]
		\item $\psi_H$ is a $k$-algebra homomorphism,\vspace{0.1cm}
		\item $\psi_H(1_S)=1_T$ if and only if $H=G$, \vspace{0.1cm}
		\item $\psi_H$ is left and right $T^H$-linear map,\vspace{0.1cm}
		\item the restriction $\psi_H|_S$ to $S$ is injective,\vspace{0.1cm}
		\item $e_H:=\psi_H(1_S)$ is a central idempotent of $T$,\vspace{0.1cm}
		\item $\psi_H(S^{\alpha_H})\subset T^H$,\vspace{0.1cm}
		\item the restriction of $\psi_H$ to $S^{\alpha_H}$ is a $k$-algebra isomorphism from $S^{\alpha_H}$ onto $T^He_H$ whose inverse is the multiplication by $1_S$. In particular $T^H1_S=S^{\alpha_H}$.
	\end{enumerate}
\end{prop}

\begin{proof}
Notice that (i) is immediate because $\beta_g$, $g\in G$, is a $k$-algebra homomorphism and $\{e_i\,:\,1\leq i\leq m\}$ is a set of orthogonal idempotents of $T$. For (ii), observe that $\beta_{h_i}(1_S)e_i=e_i$, for all $1\leq i\leq m$. Hence, $\psi_H(1_S)=e_1+\ldots +e_m$. Since $T=\sum_{g\in G}\beta_g(S)$ by (G4), it follows that $\psi_G(1_S)=1_T$. Conversely, if $\psi_H(1_S)=1_T$ then $1_T$ belongs to the ideal  $\sum\limits_{ i=1} ^m\beta_{h_i}(S)e_i$ of $T$. Hence $\sum\limits_{ i=1} ^m\beta_{h_i}(S)e_i=T=\sum_{g\in G}\beta_g(S)$,  which implies $H=G$.

The item (iii) is clear. For (iv), take $s\in S$ and notice that 
\[\psi_H(s)1_S=\sum_{i=1}^m\beta_{h_i}(s)e_i1_S=\beta_1(s)1_S=s\] 
because $e_i1_S=0$, for all $2\leq i\leq m$. Thus, (iv) follows.  Since $\{e_i\,:\,1\leq i\leq m\}$ is a set of central orthogonal idempotents of $T$ and $\psi_H(1_S)=e_1+\ldots +e_m$, the item (v) follows. 
	
For (vi), we need to show that $\beta_h(\psi_H(s))=\psi_H(s)$, for all $s\in S^{\alpha_H}$ and $h\in H$. Thus, it is enough to check that $\beta_h(\psi_H(s))$ gives us a permutation in the elements that appear in the sum of  $\psi_H(s)$ given in \eqref{fih}. First, observe that
\begin{align*}
\beta_{h_i}(1_S)\beta_{h_j}(s)&=\beta_{h_j}\left(\beta_{h_j^{-1}h_i}(1_S)s\right)=\beta_{h_j}\left(\beta_{h_j^{-1}h_i}(1_S)s1_S\right)\\
                              &\stackrel{\eqref{form1:partial-action}}{=}\beta_{h_j}(1_{h_j^{-1}h_i}s)\stackrel{(\star)}{=}\beta_{h_j}(\beta_{h_j^{-1}h_i}(s)1_S)\\
                              &=\beta_{h_i}(s)\beta_{h_j}(s)\stackrel{(\star\star)}{=}\beta_{h_j}(1_S)\beta_{h_i}(s),
\end{align*}	
where $(\star)$ follows because $\beta_h(s)1_S\stackrel{\eqref{form1:partial-action}}{=}\af_h(s1_{h^{-1}})=s1_h$ and $(\ast\ast)$ follows using that $\beta_h(1_S)$ is a central element of $T$, for all $h\in H$. Now, consider  $1\leq l\leq n$ and $i_1<\dots<i_l$ and note that
\begin{align*}
\beta_h(\beta_{h_{i_1}}(1_S)\cdots \beta_{h_{i_{l-1}}}(1_S)\beta_{h_{i_l}}(s))&=\beta_{h{h_{i_1}}}(1_S)\cdots \beta_{h{h_{i_{l-1}}}}(1_S)\beta_{h{h_{i_l}}}(s)\\
                                                                              &=\beta_{{h_{j_1}}}(1_S)\cdots \beta_{{h_{j_{l-1}}}}(1_S)\beta_{{h_{j_l}}}(s),
\end{align*}
where $h_{j_k}=hh_{i_k}$ for all $1\leq k\leq l$. Since $\beta_{h}(1_S)$, $h\in H$, are central elements of $T$, we can rearrange $\beta_{h_{j_1}}(1_S),\ldots,\beta_{h_{j_{l-1}}}(1_S)$ such that  $j_1<\ldots < j_{l-1}$. If $j_{l-1}<j_l$, there is nothing to do. Otherwise, as we showed above, $\beta_{h_{l-1}}(1_S)\beta_{h_l}(s)=\beta_{h_{l}}(1_S)\beta_{h_{l-1}}(s)$. Thus, $\beta_h(\beta_{h_{i_1}}(1_S)\cdots \beta_{h_{i_{l-1}}}(1_S)\beta_{h_{i_l}}(s))$ appears in the sum of $\psi_H(s)$ and the result follows.

Finally to (vii), notice that by (v) and (vi) we have $\psi_H(S^{\af_H})=\psi_H(S^{\af_H}1_S)\subset T^He_H$. Hence, by (1), we conclude that  $\psi_H: S^{\alpha_H} \rightarrow T^He_H$ is a $k$-algebra  homomorphism. Also, if $x\in T^H$ then
\begin{align*}
	\alpha_h((x1_S)1_{h^{-1}})&=\beta_h((x1_S)1_{h^{-1}})\stackrel{\eqref{form1:partial-action}}{=}\beta_h(x1_S\beta_{h^{-1}}(1_S))\\
	                          &=\beta_h(x1_S)1_S=\beta_h(x)\beta_h(1_S)1_S\stackrel{\eqref{form1:partial-action}}{=}x1_h\\[0.3em]
	                          &=(x1_S)1_h,
\end{align*}
for all $h\in H$. Hence  $x1_S\in S^{\alpha_H}$. Thus  $m:T^He_H\to S^{\alpha_H}$ is a $k$-algebra homomorphism well-defined, as $e_H1_S=1_S$. It is clear that $m(\psi_H(s))=s$
and $\psi_H(m(xe_H))=xe_H$, for all $s\in S^{\alpha_H}$ and $x\in T^H$.
\end{proof}

\section{Partial galois actions of quotient groups}\label{quo}
From now on in this work we will work with commutative rings with unit. Following 
\cite{DFP}, we say  that $S\supseteq S^\alpha$ is a \textit{partial Galois extension},  if 
for some $m\in \mathbb{N}$ there exist elements $x_i,y_i\in S, 1\leq i\leq m$, such that 
\begin{equation*}\label{G2}
\sum_{i=1}^mx_i\alpha_{g}(y_i1_{g^{-1}})=\delta_{1, g},\, \text{for each}\, g \in G.
\end{equation*}
The elements $x_i,y_i$  are called \textit{partial Galois coordinates} of $S$ over $S^\alpha$.

%Now we recall another notion from \cite{DFP}. Let $T$ be a subalgebra of $S.$    We set  $H_T=\{g\in G\ |\ \alpha_g(t1_{g^{-1}})=t1_g,\,\,\text{for all}\,\, t\in T\}.$ In general, $H_T$ is not a subgroup of $G$.  (See for instance   \cite[Example 6.3]{DFP}).
%Now we recall tho other notions from \cite{DFP}
%\begin{itemize}
%\item
%\begin{defn}
%Let $T$ be a subalgebra of $S.$ 
%\begin{itemize}
%\item  We set  $H_T=\{g\in G\ |\ \alpha_g(t1_{g^{-1}})=t1_g,\,\,\text{for all}\,\, t\in T\}.$ In general, $H_T$ is not a subgroup of $G$.  (See for instance   \cite[Example 6.3]{DFP} ).
%\item $T$ is said $\alpha$-strong, if for every $g,h \in G$ with $g^{-1}h \notin H_T$, and any non-zero idempotent $e\in S_{g}\cup S_{h}$ there is an element $t\in T$ such that $\alpha_{g}(1_{g^{-1}}t)e\ne \alpha_h(1_{h^{-1}}t)e$.
%\end{itemize}
%\end{defn}
%\vu

If $S$ is a $\alpha$-partial Galois extension of $S^\alpha$, then  %\cite[Theorem 5.1]{DFP} there is an one-to-one correspondence between the subgroups of $G$ and the separable subalgebras $T$ of $S$ which are $\alpha$-strong and such that $H_T$ is a subgroup of $G$ . 
\cite[Theorem 5.2]{DFP} states that for  every subgroup $H$ of $G$  one has that  $\alpha_H$-partial Galois extension of $S^{\alpha_H}$. Now we give an addendum of this result.

\begin{thm}\label{teofundpar2}
	Let $S$ be a ring, $G$ a finite group and $\alpha$ a partial action of   $G$ on $S$ such that $S$ is $\alpha$-partial Galois extension of $S^\alpha$. Then, every normal subgroup $H$ of $G$ induces a partial action $\alpha_{G/H}$ on $S^{\alpha_H}$ and $S^{\alpha_H}$ is $\alpha_{G/H}$-partial Galois extension of $S^\alpha$.
\end{thm}
\emph{From now on, we assume that 
 $H$ is a normal subgroup of $G.$}

The proof of Theorem \ref{teofundpar2} will be obtained as consequence of several results which we state and prove below. By  \cite[Theorem 2.3]{CHR}, the global action $\beta$ of $G$ on $T$ induces a global action $\beta_{G/H}$ of $G/H$ on $T^H$ in the following way: 
$$\beta_{G/H}:G/H\rightarrow \text{Aut}(T^H),\quad gH\mapsto \beta_g|_{T^H}.$$ Furthermore, if $T$ is a $G$-Galois extension of $T^G$, then $T^H$ is a $G/H$-Galois extension of $T^G=(T^H)^{G/H}$. \\
On the other hand by (v) of Proposition \ref{pro3.1} $T^He_H$ is an ideal of $T^H,$ then the action $\beta_{G/H}$ of $G/H$ on $T^H$ induces a partial action $\gamma_{G/H}$ of $G/H$ on $T^He_H$ in a canonical way, that is, $\lambda_{G/H}=(D_{gH}, \gamma_{gH})_{gH\in G/H}$ is given by
\begin{align}
\label{par1-action}&D_{gH}=(T^He_H)\cap \beta_g(T^He_H)=T^He_{gH},\,\,\,\text{where}\,\,\, e_{gH}=e_H\beta_g(e_H),&\\[.3em]
&\gamma_{gH}=\beta_g|_{D_{g^{-1}H}}:D_{g^{-1}H}\to D_{gH},\,\,\, \text{for each}\,\,\, gH\in G/H.&
\label{par2-action}\end{align}

\begin{prop} \label{lem2}With the notations above we have that $(T^H,\beta_{G/H})$ is an enveloping action of $(T^He_H,\alpha_{G/H})$.

\end{prop}
\begin{proof} By  construction, (G1), (G2) and (G3) of $\S$ 2.1 are satisfied. In order to prove (G4) it is enough to check that $T^H=\sum\limits_{i=1}^l\beta_{g_i}(T^He_H),$ where  $\mathcal{T}=\{g_1=1, g_2,\ldots,g_l\}$ is a transversal of $H$ on $G.$ Let $H=\{h_1=1,h_2,\dots , h_m\},$  and  write the elements of $G$ in the following order
\begin{equation*}
G=\{1,h_2,\dots , h_m,g_2,g_2h_2,\dots , g_2h_m\dots g_l, g_lh_2,\dots , g_lh_m\}.
\end{equation*} Since  $H$ is normal in $G$ we have that  $\bt_g(T^H)=T^H$ for all $g\in G,$ and $$I=\sum\limits_{i=1}^l\beta_{g_i}(T^He_H)=\sum\limits_{i=1}^lT^H\beta_{g_i}(e_H)$$  is an ideal of $T^H,$ we  shall check that $1_T\in I.$  %For $1\leq k \leq l$ and $g_{k}\in\mathcal{T}$ we have that $\beta_{g_{k}}(e_H)\in I.$ 
Write $f_H=\prod\limits_{i=1}^m(1_T-\bt_{h_i}(1_S))$ then $f_H\in T^H$ and   for $g\in G $ we conclude that $\bt_{g}(f_H)\in  T^H.$ % indeed for $h\in H$
%$$\bt_h(\bt_g(f_H))=\bt_h\left(\prod\limits_{i=1}^m(1_T-\bt_{gh_i}(1_S))\right)=\prod\limits_{i=1}^m(1_T-\bt_{ghh_i}(1_S))=\bt_g(f_H).$$
 Now denote by $e_{(i,j)}\in T$ the idempotent  in \eqref{losei} corresponding to the element $g_i h_j\in G,$ by \cite[P. 79]{DFP} we have that 
\begin{equation}\label{sum1t}
1_T=\sum\limits_{i=1}^{m}e_{(1,i)} +\sum\limits_{i=1}^{m}e_{(2,i)}+\dots +\sum\limits_{i=1}^{m}e_{(l,i)}.
\end{equation}
But $\sum\limits_{i=1}^{m}e_{(1,i)}=e_H\in I,$  $e_{(2,1)}=f_H\bt_{g_2}(1_S)=f_H\bt_{g_2}(e_{(1,1)})$ and  for $2\leq i\leq m$
\begin{align*}
 e_{(2,i)}&=f_H(1_T-\bt_{g_2}(1_S))(1_T-\beta_{g_2h_2}(1_S))\dots (1_T-\beta_{g_2h_{i-1}}(1_S))\beta_{g_2h_i}(1_S)\\
&=f_H\bt_{g_2}(1_T-(1_S))(1_T-\beta_{h_2}(1_S))\dots (1_T-\beta_{h_{i-1}}(1_S))\beta_{h_i}(1_S)\\
&=f_H\bt_{g_2}(e_{(1,i)}),
\end{align*}
from this we get  $\sum\limits_{i=1}^{m}e_{(2,i)}=f_H\bt_g(e_H)\in I,$
analogously $e_{(3,i)}=f_H\bt_{g_2}(f_H)\bt_{g_3}(e_{(1,i)}),$ for $1\leq i\leq m,$  and $\sum\limits_{i=1}^{m}e_{(3,i)}=f_H\bt_{g_2}(f_H)\bt_{g_3}(e_{H}),$ using this argument and the construction of the idempotents $e_{(i,j)}$ it is straighforward to see that
\begin{equation}\label{loseik}\sum\limits_{i=1}^{m}e_{(k,i)}=f_H\bt_{g_2}(f_H)\dots\bt_{g_{k-1}}(f_H) \bt_{g_k}(e_{H})\in I, \quad 2\leq k\leq  l.\end{equation}
 Using equations\eqref{sum1t} and \eqref{loseik} we get $1_T\in I,$ as desired. 
\end{proof}

\begin{cor}\label{cor-action-quotient} Let $\gamma_{G/H}$ be the partial action of $G/H$ on $T^He_H$ given by \eqref{par1-action} and \eqref{par2-action}. Then:
\begin{enumerate}[\rm (i)]
	\item  $\psi_{G/H}: T^H\to T^H$ is a $T^G$-linear homomorphism of $k$-algebras whose  restriction to $T^He_H$ is injective and such that $\psi_{G/H}(e_H)=1_T=\psi_G(1_S)$. \smallbreak
	\item $\psi_{G/H}\left( (T^He_H)^{\gamma_{G/H}} \right)\subset  (T^H)^{G/H}=T^G$. \smallbreak
	\item The restriction of $\psi_{G/H}$ to $(T^He_H)^{\gamma_{G/H}}$ is a $k$-algebra isomorphism of $(T^He_H)^{\alpha_{G/H}}$ onto $T^G=(T^H)^{G/H}$ with inverse given by the multiplication by $e_H$. Particularly, $T^Ge_H=(T^He_H)^{\gamma_{G/H}}$.\smallbreak
	\item $(T^He_H)^{\gamma_{G/H}} \subset T^He_H$ is a partial Galois extension with Galois group $G/H$ if and only if $T^G=(T^H)^{G/H}\subset T^H$ is a Galois extension with Galois group $G/H$.
\end{enumerate}		
\end{cor}

\begin{proof}
	Itens (i), (ii) and (iii) follow directly from Proposition \ref{pro3.1}, while item (iv) follows from Lemma \ref{lem2} and \cite[Theorem 3.3]{DFP}.
\end{proof}

We shall see that the partial action $\gamma_{G/H}$ on $T^He_H$ induces a partial action $\alpha_{G/H}$ of $G/H$ on $S^{\alpha_H}$ via  multiplication by $1_S$.
Indeed, %if $G/H=\{g_1H=H,\dots,g_lH\}$ then 
we define
\begin{align}
\label{idempgi}&\tilde{1}_{gH}:=e_{gH}1_S=e_H\beta_{g}(e_H)1_S=\beta_{g}(e_H)1_S,\\[.2em]
\label{dprima} &\quad \tilde{D}_{gH}:=D_{gH}1_S=T^He_{gH}1_S=S^{\alpha_H}\tilde{1}_{g},\\[.2em]
\label{alfaprima}&\qquad \alpha_{gH}:=(m_{1_S}\circ \gamma_{gH}\circ \psi_H)|_{\tilde{D}_{g^{-1}H}},
\end{align}
where $m_{1_S}:T^He_H\to S^{\alpha_H}$ is the multiplication map by $1_S$.

\begin{prop}\label{pro1.2.6.} Let $H$ be a normal subgroup of $G$, %$G/H=\{g_1H=H,\dots,g_mH\}$, 
$\tilde{D}_{gH}$ and $\alpha_{gH}$ given respectively by \eqref{dprima} and \eqref{alfaprima}. Then
\begin{enumerate}[\rm (i)]
		\item $\alpha_{G/H}=\left(\tilde{D}_{gH},\alpha_{gH}\right)_{gH\in G/H}$ is a partial action of $G/H$ on $S^{\alpha_H}$,
		\item  $(S^{\alpha_H})^{\alpha_{G/H}}=S^{\alpha}$, \smallbreak
		\item $(T^H, \beta_{G/H})$ is the globalization of $\alpha_{G/H}$,\smallbreak
		\item $S^{\alpha}\subset S^{\alpha_H}$ is a partial Galois extension with Galois group $G/H$  if and only if $T^{G}\subset T^H$ is a Galois extension with Galois group $G/H$.
	\end{enumerate}
\end{prop}

\begin{proof}
	(i) By \eqref{idempgi} we know that $\tilde{1}_{g_iH}$ is a central idempotent of $T.$ Now we check that $\tilde{1}_{gH}\in S^{\alpha_H}.$  Let   $h\in H,$ then taking into account (vi) of   Proposition \ref{pro3.1} and  the fact that $e_H=\psi_H(1_S)$ we have 
	\begin{align*}
	\af_h(\tilde{1}_{gH}1_{h\m})&=\af_h(\beta_{g}(e_H)1_{h\m})=\bt_{hg}(e_H)1_{h}
	=\bt_{gh'}(e_H)1_{h}=\beta_{g}(e_H)1_S1_h=\tilde{1}_{gH}1_h,
	\end{align*} 
	then $\tilde{D}_{g_iH}$ is an ideal of $S^{\alpha_H}$. Furthermore, 
%$$\tilde{D}_{g\m_iH}=T^H1_S\beta_{g\m_i}(e_H)1_S=B^H\beta_{g\m_i}(e_H)1_S=B^He_H\beta_{g\m_i}(e_H)1_S=B^He_{g\m_iH}1_S,$$ 
Now by (i) and (iii) of Proposition \ref{pro3.1} we get 
	\begin{equation}\label{impsi}\psi_H(S^{\alpha_H}\tilde{1}_{gH})=T^He_{g\m H},\end{equation} which is the domain of $\gamma_{gH},$ then 
	
\begin{align*}
	\alpha_{gH}(\tilde{D}_{g\m H})\stackrel{\eqref{impsi}}=m_{1_S}\circ \gamma_{gH}(T^He_{g\m H})=T^He_{gH}1_S=S^{\alpha_H}\beta_{g}(e_H)1_S\stackrel{\eqref{dprima}}=\tilde{D}_{gH}.
	\end{align*} 
	Thus $\alpha_{gH}:\tilde{D}_{g\m H}\to \tilde{D}_{g H}$ is well defined and is a ring isomorphism as it is composition of isomorphisms and  (P1) of \S 2.1. is satisfied,  now condition (P2) is clear. To check (P3) take $gH, l H\in G/H,$ then

		%\item [(P2)] $D'_1=D_H1_S=S^{\alpha_H}e_H1_S=S^{\alpha_H}1_S=S^{\alpha_H}$ and $\alpha'_1=(m_{1_S} \circ \alpha_{H}\circ \psi_H)|_{D'_1}=id_{S^{\alpha_H}}$. Now for $g,l\in \mathcal T,$ we have
	\begin{align*} 
		\alpha_{gH}(\tilde{D}_{g_{-1}H}\cap \tilde{D}_{lH})&=(m_{1_S}\circ \gamma_{gH}\circ \psi_H)(S^{\alpha_H}\tilde{1}_{g\m H}\tilde{1}_{lH})\\
		&\stackrel{\eqref{impsi}}=(m_{1_S}\circ \gamma_{gH})(T^He_{g\m H}e_{l H})\\
		%&=m_{1_S}(B^He_{gH}e_{lg H})\\
		&=T^He_{gH}1_ST^He_{l H}1_S\\
		&=\tilde{D}_{g_H}\cap \tilde{D}_{glH}.
		\end{align*}

Finally, to check (P4) take  $x\in \tilde{D}_{l\m H}\cap \tilde{D}_{(gl)\m H}$. Then,
		\begin{align*}
		\alpha_{gH}\circ \alpha_{lH}(x)&=(m_{1_S}\circ \gamma_{gH}\circ\psi_H)\circ(m_{1_S}\circ \gamma_{hH}\circ\psi_H)(x)\\
		%&=(m_{1_S}\circ \alpha_{gH})\circ id_S\circ(\alpha_{lH}\circ \psi_H)(x)\\
		&=(m_{1_S}\circ (\gamma_{gH}\circ \gamma_{lH})\circ \psi_H)(x)\\
		&\stackrel{\eqref{impsi}}=(m_{1_S}\circ \af_{glH}\circ \psi_H)(x)\\
		&=\alpha_{glH}(x).
		\end{align*}

	(ii) First of all observe that 
\begin{equation}\label{otra}\psi_H\circ \alpha_{gH}=\gamma_{gH}\circ \psi_H \quad\text{on}\quad\tilde{D}_{g\m_H} \quad\text{ for all} \,\,g\in H.\end{equation}
 Moreover since $\tilde{1}_{g\m H}\in S^{\alpha_H},$    then $\psi_H(\tilde{1}_{g\m H})=\beta_{g\m}(e_H)=\beta_{g\m}(e_H)e_H=e_{g\m H}.$ This implies $\psi_H((S^{\alpha_H})^{\alpha_{G/H}})=\psi_H(S^{\alpha_H})^{\gamma_{G/H}}.$ Indeed, let $s\in S^{\alpha_H}$ such that $ \psi_H(s)\in\psi_H(S^{\alpha_H})^{\gamma_{G/H}}.$ Then 
	\begin{align*}
	\af_{gH}(s\tilde{1}_{g\m H})&=\psi_H(\af_{gH}(s\tilde{1}_{g\m H}))1_S=\gamma_{gH}\circ \psi_H(s\tilde{1}_{g\m H}))1_S
	\\&\stackrel{\eqref{impsi}}=\gamma_{gH}(\psi_H(s)e_{g\m H})1_S=\psi_H(s)e_{g H}1_S=s\tilde{1}_{g H},
	\end{align*}
	and $s\in S^\alpha_{G/H},$ thus $\psi_H(S^{\alpha_H})^{\gamma_{G/H}}\subseteq \psi_H((S^{\alpha_H})^{\alpha_{G/H}}),$ in an analogous way one shows the other inclusion. Then 
	\begin{align*}(S^{\alpha_H})^{\alpha_{G/H}}&\stackrel{(vii)  Prop. \ref{pro3.1}}=\psi_H((S^{\alpha_H})^{\alpha_{G/H}})1_S=\psi_H(S^{\alpha_H})^{\gamma_{G/H}}1_S
	\\&
	=(T^He_H)^{\gamma_{G/H}}1_S\stackrel{iii)  Cor.\ref{cor-action-quotient}}=T^Ge_H1_S=T^G1_S=S^{\alpha},\end{align*}
	as desired.

	(iii) By (vii) of  Proposition $\ref{pro3.1}$, there is a ring monomorphism $\psi_H: S^{\alpha_H}\to  T^H,$ such that $\psi_H(S^{\alpha_H})=T^He_H$ is an ideal of $T^H.$  Moreover
	
	\begin{enumerate}
		
		\item (G2) The equality $\psi_H(\tilde{D}_{g_H})=\psi_H(S^{\alpha_H})\cap \beta_{g}(\psi_H(S^{\alpha_H})),$ follows from \eqref{impsi}.

		\item (G3) It follows by \eqref{par2-action} and \eqref{otra}  that $\psi_H\circ \alpha_{gH}=\bt_{gH}\circ \psi_H$ on $\tilde{D}_{g_H}$.

		\item  (G4) Finally using    Proposition \eqref{lem2} that
$$ T^H=\sum\limits_{gH\in G/H \mathcal T}\beta_{gH}(T^He_H)=\sum\limits_{g_i\in \mathcal T} \beta_{g_i}(\psi_H(S^{\alpha_H})),$$ 
	\end{enumerate}
	and the result follows  using  \cite[Definition 4.2]{DE}.

	(iv) This follows from (iii) and  \cite[Theorem 3.3]{DFP}.
\end{proof}

	{\it Proof of Theorem \ref{teofundpar2}.}  This is a consequence of  iv) of  Proposition \ref{pro1.2.6.} and  item (ii) of   \cite[Theorem 2.3]{CHR}.

\section{Enveloping actions  and partial $G$-isomorphisms}\label{enve}

In this section we will fix a commutative ring $R$  and work with partial Galois extensions of $R$ with Galois group $G.$

The trace map plays an important role when having partial  actions of finite groups on algebras, and for a partial action $(S,\alpha)$ is defined as follows $tr_{S/R}(s)=\sum\limits_{g \in G}\alpha_{g}(s1_{g^{-1}}),$ for all $s\in S,$ by \cite[Lemma 2.1]{DFP} $tr_{\alpha}\colon S\to R$  is a $R$-linear map. % moreover from \cite[Remark 3.4]{DFP}, in the ring $S$ there exists $w\in S$ such that
%\begin{equation}\label{tr1}tr_{\alpha}(w)=1.\end{equation}

\begin{defn}\label{pariso}
	We say that two partial Galois extensions $(S, \alpha)$ and $(S',\alpha')$ of $R$ are called partially \emph{$G$-isomorphic}, and denoted $(S,\alpha)\overset{par}{\sim} (S',\alpha')$, if there is a $R$-algebra isomorphism $f: S\rightarrow S'$ such that for all $g \in G$:
	\begin{enumerate}
		\item [(i)] $f(S_g)\subseteq S'_g$,
		\item [(ii)]$(f \circ \alpha_g)|_{S_{g^{-1}}} = (\alpha'_g \circ f)|_{S_{g^{-1}}}$.
	\end{enumerate}
	
\end{defn}
The relation $\overset{par}{\sim}$ defined above is an equivalence relation. We denote by $[S,\alpha]$ the equivalence class of $(S,\alpha)$ relative to $\overset{par}{\sim}.$
Next result implies that  we may only require  that the map $f$  to be a $R$-algebra homomorphism. %definition above, we can require less than $R$-algebras isomorphism that preserve partial action. It is enough if there is a $R$-algebra homomorphism that commutes with this partial action, that is,

\begin{prop}\label{pro9}
	Let $(S,\alpha)$ and $(S',\alpha')$ be partial Galois extension of $R.$ Then $(S, \alpha)$ and $(S',\alpha')$  are partially \emph{$G$-isomorphic}, if and only if, there is an $R$-algebra homomorphism $f: S\rightarrow S'$ satisfying {\rm (i)} and {\rm (ii)} above.
	%\begin{enumerate}
	%\item [(i)] $f(S_g)=S'_g$,
	%\item [(ii)] $(f\circ \alpha_g)|_{S_{g^{-1}}}=(\alpha'_g \circ f)|_{S_{g^{-1}}}$.
	%\end{enumerate}
	%Then $f$ is an isomorphism of $R$-algebras.
\end{prop}
\begin{proof}
	Note that $1_S=1_R=1_{S'}$. Let $x_i, y_i$, $1\leq i \leq n$ be partial Galois coordinates of $S$ over $R$ relative to $\alpha$.
	Take $s'\in S'$ then:
	\begin{align*}
	f\left(\sum_{i=1}^nx_itr_{\alpha'}(f(y_i1_{g^{-1}})s')\right)&=\sum_{i=1}^nf(x_i)tr_{\alpha'}(f(y_i1_{g^{-1}})s')\\
	%&=\sum_{i=1}^nf(x_i)\sum_{g \in G}\alpha'_g(f(y_i1_{g^{-1}})s'1'_{g^{-1}})\\
	&=\sum_{i=1}^nf(x_i)\sum_{g \in G}\alpha'_g(f(y_i1_{g^{-1}}))\alpha'_g(s'1'_{g^{-1}})\\
	&=\sum_{i=1}^nf(x_i)\sum_{g \in G}(\alpha'_g\circ f)(y_i1_{g^{-1}})\alpha'_g(s'1'_{g^{-1}})\\
	&=\sum_{i=1}^nf(x_i)\sum_{g \in G}(f\circ \alpha_{g})(y_i1_{g^{-1}})\alpha'_g(s'1'_{g^{-1}})\\
	&=\sum_{g \in G}f\left(\sum_{i=1}^nx_i\alpha_{g}(y_i1_{g^{-1}})\right)\alpha'_g(s'1'_{g^{-1}})\\
	&=\sum_{g \in G}\delta_{1,g}1_S\alpha'_g(s'1'_{g^{-1}})=s'1_S=s',
	\end{align*}
	and $f$ is surjective.

	To prove that $f$ is injective take  $s\in S$ be such that $f(s)=0,$ we have
	
	\begin{align*}
	f(\alpha_{g}(y_is1_{g^{-1}}))&=(f\circ \alpha_{g})(y_is1_{g^{-1}})=(\alpha'_g\circ f)(y_is1_{g^{-1}})\\
	&=\alpha_g'(f(y_i)f(s)f(1_{g^{-1}}))=0,
	\end{align*}
	for all $g \in G$ and 
	$
	tr_\alpha(y_is)=f(tr_\alpha(y_is)1_S)=0,\,\, \text{for all}\,\, 1\leq i\leq n.
	$
	Then,
	\begin{align*}
	0&=\sum_{i=1}^nx_itr_\alpha(y_is)=\sum_{g \in G}\left(\sum_{i=1}^nx_i\alpha_{g}(y_i1_{g^{-1}})\right)\alpha_{g}(s1_{g^{-1}})\\
	&=\sum_{g \in G}\delta_{1,g}\alpha_{g}(s1_{g^{-1}})=s1_S=s,
	\end{align*}
	and $f$ is injective.
\end{proof}

Given two (global) actions $(T^1,\beta^1)$ and $(T^2,\beta^2),$ with  $(T^1)^G=A=(T^2)^G$ we recall from  \cite{H}
that $(T^1,\beta^1)$ and $(T^2,\beta^2)$ are  $G$-isomorphic if there is an $A$-algebra homomorphism $f:T^1\to T^2$ such that $f\circ\beta=\beta'\circ f$.

For our purposes, the concept of (global) $G$-isomorphism  for enveloping actions of partial action needs one more condition than the classical concept  of $G$-isomorphisms defined in \cite{H}.

\begin{defn}\label{globiso}
	Let $(T,\beta)$ and $(T',\beta')$ be enveloping actions of $(S,\alpha)$ and $(S',\alpha')$ respectively,  suppose that $T^G=A=T'^G$ . 
We say that $(T,\beta)$ and $(T',\beta')$ are globally $G$-isomorphic, and we denote $(T,\beta)\overset{gl}{\sim}(T',\beta')$, if they are $G$-isomorphic and  the map $f:T\to T'$ giving the $G$-isomorphism satisfies $f(1_S)=1_{S'}.$
	
	%Then $\overset{gl}{\sim}$ is an equivalence relation, and the equivalence class of $(T,\beta)$ is denoted by $[T, \beta].$
\end{defn}

\begin{rem}\label{obsglobiso}
The relation $\overset{gl}{\sim}$ defined above is an equivalence relation, and the equivalence class of $(T,\beta)$ is denoted by $[T, \beta].$
\end{rem}

%We have the following.
%\begin{prop}\label{envol}
%	Let $(S,\alpha), (S', \alpha')$ be partial Galois extension of $R$ with enveloping actions $(T, \beta)$ and $(T',\beta')$ respectively. If $(T,\beta)$ and $(T',\beta')$ are  $G$-isomorphic, then $(T',\beta')$ (resp., $(T,\beta)$) is an enveloping action of $(S, \alpha)$ (resp., $(S',\alpha')$).
%\end{prop}
%
%\begin{proof}
%	Let $f:T\to B'$ and $\varphi:S\to T$ be the $G$-isomorphism and the ring monomorphism given in definition of enveloping action given in \cite[Definition 4.2]{DE}, respectively. To check that $(T',\beta')$ is an enveloping action of  $(S, \alpha)$, observe that $f\circ\varphi \colon S\to B'$ is a $k$-algebra monomorphism, and it is not difficult to check that  conditions  (i')-(iii') in \cite[p. 1938]{DE} are satisfied. 
%
%\end{proof}

Now we give the main result of this section.
\begin{thm}\label{pro10}
	Assume that $(S,\alpha), (S', \alpha')$ are partial Galois extension of $R$ and, $(T, \beta)$ and $(T',\beta')$ be their enveloping actions, respectively. Let $H$ be a normal subgroup of $G,$ if $(T,\beta)$ and $(T',\beta')$ are globally $G$-isomorphic, then $(S^{\alpha_H},\alpha_{G/H})$ and $(S'^{\alpha'_H},\alpha'_{G/H})$ are partially $G/H$-isomorphic. In particular, $(S,\alpha)$ and $(S',\alpha')$ are partially $G$-isomorphic.
\end{thm}
\begin{proof} 
	Let $H=\{h_1=1,h_2,\ldots,h_m\}$ and  $\mathcal T=\{g_1=1,g_2,\dots,g_l\}$ be a tansversal of $H$ in $G$. Consider an $A$-algebra isomorphism $f: T\rightarrow T'$ such that
	\begin{align}\label{G-iso3}
	\begin{split}
	f(1_S)&=1_{S},\\
	f \circ \beta_g&=\beta'_g\circ f,\,\,\,\text{for all}\,\,\, g \in G.
	\end{split}
	\end{align}
	It is clear that $f$ induces (by restriction to $B^H$) an isomorphism between the $A$-algebras $T^H$ and $T'^H$. By  Proposition \ref{lem2} $(T^H,\beta_{G/H})$ is an enveloping action of $(T^He_H,\alpha_{G/H})$, where $e_H=\psi_H(1_S)=\sum_{j=1}^m\beta_{h_j}(1_S)e_j,$ and the family $\{e_i\}_{1\leq i\leq m}$ is given by \eqref{losei}. %$$e_1=1_S\,\,\,\,\text{and}\,\,\,\, e_j=(1_B-1_S)\cdots(1_B-\beta_{h_{j-1}}(1_S))\beta_{h_j}(1_S),\,\, 1\leq j\leq n.$$
	
	In the same way  $(T'^H,\beta'_{G/H})$ is an enveloping action of $(T'^He'_H,\theta_{G/H})$, where $e'_H=\sum_{j=1}^m\beta'_{h_j}(1_{S})e'_j,$ and
	$$e'_1=1_{S}\,\,\,\,\text{and}\,\,\,\, e'_j=(1_{B'}-1_{S})\cdots(1_{B'}-\beta'_{h_{j-1}}(1_{S}))\beta'_{h_j}(1_{S}),\,\, 1\leq j\leq n.$$
	By (\ref{G-iso3}) we have that $f(e_j)=e'_j,$ for all $1\leq j\leq n,$
	%\begin{align*}
	% f(e_j)&=f((1_B-1_S)(1_B-\beta_{h_{j_2}}(1_S))\cdots(1_B-\beta_{h_{j-1}}(1_S))\beta_{h_j}(1_S))\\
	%&=(f(1_B)-f(1_S))(f(1_B)-f(\beta_{h_{j_2}}(1_S)))\cdots (f(1_B)-f(\beta_{h_{j-1}}(1_S)))f(\beta_{h_j}(1_S))\\
	%&=(1_{B'}-1_{S'})(1_{B'}-\beta'_{h_{j_2}}(1_{S'}))\cdots(1_{B'}-\beta'_{h_{j-1}}(1_{S'}))\beta'_{h_j}(1_{S'})\\
	%&=e'_j.
	%\end{align*}
	and consequently $f(T^He_H)=T'^He'_H,$
	%\begin{align*}
	%f(B^He_H)&=f(B^H)f(e_H)=f(B^H)f(\sum_{j=1}^n\beta_{h_j}(1_S)e_j)\\
	%&=f(B^H)(\sum_{j=1}^nf(\beta_{h_j}(1_S))f(e_j))\\
	%&=B'^H(\sum_{j=1}^n\beta'_{h_j}(f(1_S))e'_j)\\
	%&=B'^H(\sum_{j=1}^n\beta'_{h_j}(1_{S'})e'_j)\\
	%&=B'^He'_H.
	%\end{align*}
	in particular, $f(e_H)=e'_H.$ %since $1_B \in B^H$.

	On the other hand, it follows from \eqref{dprima} that  %\ref{pro1.2.6.}, one have that  the corresponding  ideals  of  the partial actions $(\alpha', S^{\alpha_H})$ and $(\theta', S'^{\alpha'_H})$ of $G/H$ on $S^{\alpha_H}$ and $S'^{\alpha'_H},$ respectivel are
	$\tilde{D}_{gH}=T^He_H\beta_{g}(e_H)1_S\quad\text{and}\quad \tilde{D'}_{gH}=T'^He'_H\beta'_{g}(e'_H)1_{S'},$ and the partial isomorphisms are  
	$$\alpha_{gH}=(m_{1_S}\circ\gamma_{gH}\circ\psi_H)|_{\tilde{D}_{g\m H}}\quad\text{and}\quad \alpha'_{gH}=(m_{1_S}\circ\gamma'_{gH}\circ\psi'_H)|_{\tilde{D'}_{g\m H}},$$ respectively,   for all $g\in G.$ %where $\psi'_H:B'\to B'$  is the ring homomorphism $\psi'_H(b')=\sum_{1\leq i\leq n}\beta'_{h_i}(b')e'_i$ for all $b'\in B'$.
	Now the  map $\varphi=m_{1_{S}}\circ f\circ\psi_H:S\to S'$ %$$\varphi=m_{1_{S}}\circ f\circ\psi_H:S\xrightarrow{\psi_H} B\xrightarrow{f} B'\xrightarrow{1_{S'}} S'$$
	is a ring homomorphism, we shall check that it is $R$-linear. By  item (ii) of Proposition \ref{pro3.1} we have $A1_S=R.$ Then, for $r=a1_S\in R$ and $s\in S$ we have that
	\begin{align*}
	\varphi(rs)&=(m_{1_{S}} \circ f \circ \psi_H)(a1_Ss)=(m_{1_{S}} \circ f)(a \psi_H(1_S)\psi_H(s))\\
	&=(m_{1_{S}}\circ f)(a e_H \psi_H(s))=1_S(af(1_B)f(e_H)f(\psi_H(s)))\\
	&=1_{S}(a1_{B'}e'_Hf(\psi_H(s)))=1_S(ae'_Hf(\psi_H(s)))\\
	&=(a1_S)(1_{S}e'_H)f(\psi_H(a))=r(1_Sf(\psi_H(s))\\
	&=r(m_{1_{S}} \circ f \circ \psi_H)(s)\\
	&=r\varphi(s).
	\end{align*}
	
	Now, we check that $\varphi$ satisfies conditions (i) and (ii) of Definition \ref{pariso}. Let $g\in G.$ Then
	\begin{align*}
	\varphi(\tilde{D}_{gH})&%=(m_{1_{S}} \circ f \circ \psi_H)(S^{\alpha_H}1'_{g})\\
	\stackrel{\eqref{impsi}}=m_{1_{S}} \circ f(T^He_{gH})\\
	&=m_{1_{S}} (T'^Hf(e_{g_iH}))\\
	&=(m_{1_{S}} \circ f)(T^H\beta_{g}(e_H)e_H)\\
	&=1_{S}B'^Hf(\beta_{g}(e_H))f(e_H)\\
	&=1_{S}T'^H\beta'_{g}(f(e_H))e'_H\\
	&=T'^He'_H\beta'_{g}(e'_H)1_{S}\\
	&=\tilde{D'}_{gH}.
	\end{align*}
	Thus condition (i) is satisfied. Now we check (ii). Let $x\in \tilde{D}_{g\m H}.$ Then
	\begin{align*}
	(\varphi \circ \alpha_{gH})(x)&%=[(m_{1_{S}}\circ f \circ \psi_H)\circ (m_{1_{S}} \circ \alpha_{g_iH} \circ \psi_H)](x)\\
	%&=[(m_{1_{S}}\circ f \circ (\psi_H \circ m_{1_{S}}) \circ \alpha_{g_iH} \circ \psi_H)](x)\\
	%&=[(1_{S'}\circ f \circ id_B\circ \alpha_{g_iH} \circ \psi_H)](x)\\
	=(m_{1_{S}}\circ (f \circ \gamma_{gH}) \circ \psi_H)(x)\\
	&\stackrel{\eqref{G-iso3}}=(m_{1_{S}}\circ (\gamma'_{gH} \circ f) \circ \psi_H)(x)\\
	&=[(m_{1_{S}}\circ \gamma'_{gH} \circ id_{B'}\circ f \circ \psi_H)](x)\\
	&=[(m_{1_{S}}\circ \gamma'_{gH} \circ (\psi'_H \circ m_{1_{S}})\circ f \circ \psi_H)](x)\\
	&=[(m_{1_{S}}\circ \gamma'_{gH} \circ \psi'_H) \circ (m_{1_{S}}\circ f \circ \psi_H)](x)\\
	&=(\alpha'_{gH}\circ \varphi)(x).
	\end{align*}
	Finally,  by  Proposition \ref{pro9} $(S^{\alpha_H},\alpha_{G/H})$ and $(S'^{\theta_H},\alpha'_{G/H})$ are partially $G/H$-isomorphic.
	For the last affirmation, it is enough to take $H=\{1\}$.
\end{proof}

\section{The Construction of $\Hh_{par}(G,R)$}
From now on in this work, $G$  will denote a finite abelian group.

Let ${\mathcal H}_{par}(G,R)$ the set of equivalence classes of \emph{partial   abelian extensions }of $R$ with group $G$, that is partial Galois actions $(S,\alpha)$  of $R$ with  group $G.$ In this section we construct a product in  ${\mathcal H}_{par}(G,R)$ which turns it into a commutative  inverse semigroup.  First we recall the classical construction of the Harrison group (see \cite{H}).

\subsection{The Harrison group $\Hh(G,A)$}\label{harglo}
Let  $A$ be a commutative ring with unit. In \cite{H} the author introduced and studied the abelian group $\Hh(G,A)$ whose elements are equivalence  classes of (global) $G$ -equivalent Galois extensions  of $A.$  We denote the equivalence class of $(T, \beta)$  by ${\rm cl}(T,\bt)$  the  Multiplication in $\Hh(G,A)$ Is defined in the following way:\\
 Let ${\rm cl}(B,\bt), {\rm cl}(B', \bt')\in \Hh(G,A)$. It is well known that the tensor product $B\otimes_A B'$ is an abelian extension of $A$ with Galois group $G\times G$. By \cite[Theorem 2.2]{CHR} we have that $(T\otimes T')^{\delta G}\supseteq A$ is a Galois extension with Galois group $G=G\times G/\delta G,$ where $\delta G=\{(g,g\m)\mid g\in G\}.$
The group $G$ acts on $(T\otimes T')^{\delta G}$ via
\begin{equation*}
\Lambda_{(G\times G)/\delta G}: \left(g, \sum_i t_i\otimes t'_i\right)\mapsto \sum_i \bt_g(t_i)\otimes t'_i=\sum_i b_i\otimes \bt'_g(t'_i),
\end{equation*}
for all $t_i\in B$, $t'_i\in B'$ and $g\in G$. We set  
\begin{equation}\label{harprod}{\rm cl}(B,\bt) {\rm cl}(B', \bt')={\rm cl}\left((B\otimes B')^{\delta G}, \Lambda_{(G\times G)/\delta G}\right).\end{equation}
The identity element of  ${\rm cl}(E, \rho)$ containing the ring $E_G(A)$ constructed
in the following way. We choose symbols $\{e_g\}_{g\in G}$ and we consider the free $A$-module  $E_G(A)=\bigoplus_{g\in G} Ae_g,$ with basis  $\{e_g\}_{g\in G}.$ For the basis elements we define the product $e_ge_h=\delta_{g,h}e_g.$ Then $E_G(A)$ is an $A$-algebra. The action of $\rho$  of $G$ on $A$ is given by the $A$-linear extension of  $\rho_g\cdot e_{h}=e_{gh},$ for all $g,h\in G.$

Now we indicate the construction of the inverse element in $\Hh(G,A).$ %
Let ${\rm cl}(T,\bt) \in \Hh(G,A)$, the element ${\rm cl}(T,\bt)\m $ is represented by $T$ with action of $G$ defined by $\bt'_g(t)=\bt_{g^{-1}}(t)$, for all $g\in G$ and $t \in T$.

It is shown in \cite{H} that the study of the group $\Hh(G,A)$ reduces to the
case of cyclic Galois groups.
\subsection{The  group $\Hh_{gl}(G,A)$} 

%\begin{defn}\label{globiso}
%	Let $(T,\beta)$ and $(T',\beta')$ be enveloping actions of $(S,\alpha)$ and $(S',\alpha')$ respectively,  suppose that $T^G=A=T'^G$ . 
%We say that $(T,\beta)$ and $(T',\beta')$ are globally $G$-isomorphic, and we denote $(T,\beta)\overset{gl}{\sim}(T',\beta')$, if they are $G$-isomorphic and  the map $f:T\to T'$ giving the $G$-isomorphism satisfies $f(1_S)=1_{S'}.$
	
%	Then $\overset{gl}{\sim}$ is an equivalence relation, and the equivalence class of $(T,\beta)$ is denoted by $[T, \beta].$
%\end{defn}

From Definition \ref{globiso} and Remark \ref{obsglobiso} we can to consider the group $\Hh_{gl}(G,A)$ consists of all equivalence classes $[T,\beta]$ of the  relation $\overset{gl}\sim$ defined on the set of enveloping actions, with fixed part $A,$ of partial Galois extensions $S\supseteq R.$ We define a  product $*_{gl}$ in $\Hh_{gl}(G,A)$ by the formula \eqref{harprod}, that is $$[T,\beta]*_{gl}[T',\beta']
=[(T\otimes T')^{\delta G},\Lambda_{(G\times G)/\delta G}],$$
for all $[T,\beta], [T',\beta'] \in  \Hh_{gl}(G,A).$ Then we have the following.
\begin{prop} \label{hglo}The set $\Hh_{gl}(G,A)$ is a group, and the map $\om : \Hh_{gl}(G,A) \ni[T,\beta]_{gl} \mapsto \cl($T$, \beta)\in  \Hh(G,A)$ is a group homomorphism.
\end{prop}
\begin{proof}  To prove that $[E_G(A), \rho]\in \Hh_{gl}(G,A) $ we notice that $(E_G(A), \rho)$ is an enveloping action of $(E_G(R), m_{1_S}\circ \rho).$ Indeed,   by (vii) of Proposition \ref{pro3.1}  the map 
	$$\Psi_G: E_G(R) \ni  \sum\limits_{g\in G}xe_g\mapsto \sum\limits_{g\in G}\psi_G(x)e_g\in E_G(A),$$s a  $k$-algebra isomorphism. Now for $g,h\in G$ and   $xe_h\in E_G(R)$ we have
	$$\rho_g\circ \Psi_G(xe_h)=\rho_g(\psi_G(x)e_h)=\psi_G(x)e_{gh}=\Psi_G\circ (m_{1_S}\circ \rho_g)(xe_h).$$ 
	Moreover, it is clear that that $\sum\limits_{g\in G}\rho_g(\Psi_G(E_G(R)))=E_G(A),$
	and thus $[E_G(A), \rho]$ is the identity element of $\Hh_{gl}(G,A).$ 
	
	To check that $\Hh_{gl}(G,A)$ is closed under products, take $[T,\beta], [T',\beta'] \in \Hh_{gl}(G,A),$ where $(T,\beta)$ and $(T',\beta')$  are the enveloping actions of the partial Galois extensions $(S,\alpha)$ and $(S',\alpha')$ of $R.$ By Proposition \ref{pro1.2.6.}, $(S\otimes S')^{{\alpha\otimes \alpha'}_{\delta G}}$ is a $(\alpha\otimes\alpha')_{(G\times G)/\delta G}$-partial abelian extension of $R.$ Since $(g,h)=(g,1) \pmod{\delta G}$ for all $(g,h)\in G\times G,$ one has  by Proposition \ref{pro1.2.6.} that $((B\otimes B')^{\delta G},(\beta\otimes \beta')_{(G\times G)/\delta G})$ is an enveloping action of $((S\otimes S')^{\alpha \otimes \alpha'_{\delta G} },(\alpha\otimes\alpha')_{(G\times G)/\delta G}),$ where
	\begin{align*}
	\tilde{D}_{(g,1)\delta{G}}&=(T\otimes T')^{\delta G}(\beta_{g}\otimes \beta'_{1})(e_{\delta G})1_S\otimes 1_{S},\\
	(\alpha\otimes\alpha')_{G\times G/\delta G}&=(\tilde{D}_{(g, 1)\delta_G}, (\gamma \otimes \gamma')_{(g,1)\delta G})_{g\in G},\\
	\alpha \otimes \alpha'_{(g,1)}&=(m_{1_S\otimes 1_{S}}\circ \alpha \otimes \alpha'_{(g,1)\delta G}\circ \psi_{\delta G}) |_{\tilde{D}_{(g\m,1)\delta_{G}}},
	\end{align*} then $[(T\otimes T')^{\delta G},(\beta\otimes \beta')_{(G\times G)/\delta G})]\in \Hh_{gl}(G,A)$ and $\Hh_{gl}(G,A)$ is closed under products.
	
	To check that $ \Hh_{gl}(G,A)$ is closed under inverse elements, consider $\alpha^\star$ the partial action of $G$ on $S^\star=S$, with ideals $S^\star_g=S_{g^{-1}}$ and partial isomorphisms $\alpha^\star_g:S^\star_{g^{-1}}\to S^\star_g$ given by $\alpha^\star_g=\alpha_{g^{-1}}$ for all $g\in G$. We denote by $\beta^\star$ the global action of $G$ on $T^\star=T$ given by $\beta^\star_g=\beta_g^{-1}$ for all $g\in G$. Note that $(S^\star,\alpha^\star)$ is a partial abelian extension of $R$,
	and $(B^\star,\beta^\star)$ is an enveloping action for $(S^\star,\alpha^\star),$ with  $(T^\star)^G=A.$ Then  $[T^\star,\beta^\star]=[T,\beta]^{-1}$ in $\Hh_{gl}(G,A)$.
	Finally it is clear that $w$ is a group homomorphism. 
\end{proof}
%
%Since the map
%\begin{align*}
%\theta: B\otimes B&\rightarrow \sum_{l\in G}Be_l\\
%b\otimes b'&\mapsto \sum_{l\in G}bl(b')e_{l},
%\end{align*}
%is a $B$-algebras isomorphism \cite[Theorem 1.3(e)]{CHR}. Furthermore, $B \otimes B$ and $\sum_{l\in G}Be_{l}$ are abelian extension of $A$ with Galois group $G \times G$ acting, respectively,  on the $A$-algebras 
%\begin{align*}
%(g,h)(b\otimes b')&=g(b)\otimes h^{-1}(b'),\\
%(g,h)(\sum_{l\in G}b_{l}e_l)&=\sum_{l\in G}g(b_{l})e_{ghl},
%\end{align*}
%for all $g,h\in G$. Then, $\theta\circ(g,h)=(g,h)\circ\theta$ for all $g,h\in G$  and hence $\theta$ induces by restriction a $A$-algebras isomorphism between $(B\otimes B')^{\delta G}$ and $E=\sum_{g\in G}Ae_{g}=(\sum_{g \in G}Be_{g})^{\delta G}$. Therefore, $[B]\ast[B]^{-1}=[E]$ and the required follows.
\subsection{The inverse semigroup $\Hh_{par}(G,R)$}\label{harsemi}

Consider $[S,\alpha],[S',\alpha']\in \Hh_{par}(G,R)$,  then by \cite[Proposition 2.9]{DPP}  we have that  $S\otimes S'$ is a $\alpha\otimes \alpha'$- partial abelian extension of $R\otimes R=R$. 
We define on $\Hh_{par}(G,R)$ the operation $\ast_{par}$ by
\begin{align}\label{harsemprod}
[(S,\alpha)]\ast_{par}[(S,\alpha')]%=[1_S\otimes1_{S'}(B\otimes B')^{\delta G},1_S\otimes1_{S'}\circ\Lambda_{(G\times G)/\delta G}]\\
&=[(S\otimes S')^{{\alpha \otimes \alpha'}_{\delta G}}, (\alpha\otimes\alpha')_{(G\times G)/\delta G}],
\end{align} where  $(\alpha\otimes\alpha')_{(G\times G)/\delta G}$ is given by \eqref{idempgi},
\eqref{dprima} and \eqref{alfaprima}.
%with $\Lambda=\beta\otimes\beta'$ and $\lambda=\alpha\otimes\alpha',$ 
%This product is well defined thanks to  Proposition \ref{hglo} and Theorem \ref{pro10}.

Before proving that the product  \eqref{harsemprod} is well defined  we present a  description of  the idempotents $\tilde 1_{gH}=\beta_{g}(e_H)1_S, g\in G$  given in \eqref{idempgi} and the maps $\af_{gH}$ in \eqref{alfaprima}. Let $ H=\{h_1=1, h_2, \dots, h_m\}.$ Then
   \begin{align*}\tilde 1_{gH}&=\sum_{i=1}^m\beta_{gh_i}(1_S)\beta_{g}(e_i)1_S=\sum_{i=1}^m\beta_{gh_i}(1_S)1_S\beta_{g}(e_i)1_S=\sum_{i=1}^m1_{gh_i}\beta_{g}(e_i)1_S
\\&\stackrel{\eqref{btgei}}=1_g+\sum_{i=2}^m1_{gh_i}\prod\limits_{j=2}^i(1_S-1_{gh_{j-1}})1_{gh_i}=
1_g+\sum_{i=2}^m\prod\limits_{j=2}^i(1_S-1_{gh_{j-1}})1_{gh_i}.
\end{align*} Then
\begin{equation}\label{1gh}
\tilde 1_{gH}=1_g+\sum_{i=2}^m\prod\limits_{j=2}^i(1_S-1_{gh_{j-1}})1_{gh_i}.
\end{equation}
Now by \eqref{alfaprima} we have for $g\in G$ and $x\in \tilde D_{g^{-1}H}$ that
\begin{align*} \alpha_{gH}(x)&=(m_{1_S}\circ \gamma_{gH}\circ \psi_H)(x)
=\sum_{i=1}^m\beta_{gh_i}(x)\beta_{g}(e_i)1_S
\\&=\sum_{i=1}^m\af_{gh_i}(x1_{(gh_i)\m})\beta_{g}(e_i)1_S
\\&=\af_{g}(x1_{g\m})+ \sum_{i=2}^m\af_{gh_i}(x1_{(gh_i)\m})\prod\limits_{j=2}^i(1_S-1_{gh_{j-1}}).
\end{align*}
That is 
\begin{equation}\label{alphapprima}  \alpha_{gH}(x)=\af_{g}(x1_{g\m})+ \sum_{i=2}^m\prod\limits_{j=1}^{i-1}(1_S-1_{gh_{j}})\af_{gh_i}(x1_{(gh_i)\m}).
\end{equation}

\begin{lem} \label{welldef}Let $(S,\alpha)$ and $(S',\alpha')$ partial Galois extension of $R$ with the same group $G.$ If $H$ is a subgrou of $G$ then $(S^H,\alpha_{G/H})$ and $(S'^H,\alpha'_{G/H})$ are $G/H$-equivalent.
\end{lem}
\begin{proof} Let $f\colon S\to S'$ be the $G$-isomorphism. Then it is clear that $f$ restricts to a $R$-algebra isomorphism $f_{\mid S^H}\colon S^H\to S'^H$ and  the equality $f(\tilde{1}_{gH})=\tilde{1'}_{gH}$ follows by \eqref{1gh}, then $f(\tilde{D}_{gH})=\tilde{D}_{gH}$  thanks to \eqref{dprima}. Finally using the fact that $f$ is a $G$-isomorphism and \eqref{alphapprima} we get that $f_{\mid S^H}$ is a $G/H$-isomorphism which implies the result.
\end{proof}
\begin{prop}\label{boadef} The set $\Hh_{par}(G,R)$ with product
	$\ast_{par}$ is a commutative semigroup.
\end{prop}
\begin{proof} First of all by Lemma \ref{welldef} the product   $\ast_{par}$ is well defined and the fact that $\Hh_{par}(G,R)$ is closed under product follows  from Theorem \ref{teofundpar2}. Finally it is clear that with this product $\Hh_{par}(G,R)$  is commutative and associative.
\end{proof}

We recall that an inverse semigroup $\I$ is a semigroup in which the following conditions hold.
\begin{itemize}
	\item $\I$ is regular, that is,  given $x\in \I$ there is an  element $x^\star\in \I$ such that $xx^\star x=x\quad\text{and}\quad x^\star xx^\star=x^\star.$
	\item The idempotents of $\I$ commute.
\end{itemize}
It is well known that the two conditions above are equivalent to the fact that, for any $x\in \I$ there exists a unique $x^\star\in \I$ such that $xx^\star x=x\quad\text{and}\quad x^\star xx^\star=x^\star.$ The element $x^*$ is called the {\it  inverse} of $x.$ One can ckeck that 
$$E(\I)=\{xx^*\mid x\in \I\}$$   is the set of idempotents of $\I,$ and  $E(\I)$ is a meet semilattice with respect to the partial ordering $e\leq f,$ if $ef=e.$  (For more details on inverse semigroups the interested reader may consult \cite{L}).

We have the following.

\begin{thm}\label{invpar}
	$\Hh_{par}(G,R)$ is an abelian inverse semigroup, which contains $\Hh(G,R).$
\end{thm}

\begin{proof} By Proposition \ref{boadef} is a $\Hh_{par}(G,R)$ is a commutative semigroup. Thus we only need to show that $\Hh_{par}(G,R)$ is regular.
	Let $(S,\alpha)$ be a partial abelian extension of $R$ with enveloping action $(B,\beta)$. We denote by $\alpha^\star$ the partial action of $G$ on $S^\star=S$ defined  in the proof of Proposition \ref{hglo} and  $(B^\star,\beta^\star)$ an enveloping action  $(S^\star,\alpha^\star).$ We shall check that  $[S^\star,\alpha^\star]=[S,\alpha]^\star.$  Note that $[B^\star,\beta^\star]=[B,\beta]^{-1}$ in $\Hh_{gl}(G,A),$ thus $[B,\beta]\ast[B^\star,\beta^\star]\ast[B,\beta]=[B,\beta]$ in $\Hh_{gl}(G,A).$ Notice that $[B,\beta]\ast[B^\star,\beta^\star]\ast[B,\beta]=[(B\otimes B^\star\otimes B)^{\delta G}, \beta_g\otimes\beta^\star\otimes\beta],$ where  $\Delta G=\{((g^{-1},1)\delta G, g)\ |\ g\in G\}.$ Moreover the map . %It is follows that the map
	\begin{align*}
	\theta:(B\otimes B^\star\otimes B)^{\delta G}&\rightarrow B\\
	x\otimes y\otimes z&\mapsto xyz,
	\end{align*}
	is a  $A$-algebra homomorphism such that $\theta\circ(\beta_g\otimes\beta^\star_{g}\otimes\beta_{g})=\beta_g\circ\theta$ for all $g \in G$. Note that  $\theta(1_R\otimes 1_R\otimes 1_R)=1_R,$ it follows by Proposition \ref{pro9} that $\theta$ is an isomorphism and thus   $((B\otimes B'\otimes B)^{\Delta G}, (\beta\otimes\beta'\otimes\beta)_{(G\times G\times G)/\Delta G})\quad\text{and}\quad (B,\beta),$  are globally $G=(G\times G\times G)/\Delta G$-isomorphic. Since these extensions are  enveloping actions of $$((S\otimes S^\star\otimes S)^{(\alpha\otimes\alpha^\star\otimes\alpha)_{\Delta G}},(\alpha\otimes\alpha^\star\otimes\alpha)'_{(G\times G\times G)/\Delta G})\quad\text{and}\quad (S,\alpha),$$ 
respectively, then by  Theorem \ref{pro10} these partial actions are $(G\times G\times G)/\Delta G$-isomorphic, that is, $[S,\alpha]\ast[S,\alpha]^\star\ast[S,\alpha]=[S,\alpha]$  In an analogous way one shows that
	$[S,\alpha]^\star\ast[S,\alpha]\ast[S,\alpha]^\star=[S,\alpha]^\star, $ and we conclude that $\Hh_{par}(G,R)$ is an inverse semigroup. Finally it is clear that  $\Hh(G,R)$  is a subgroup of $\Hh_{par}(G,R).$
\end{proof}

Because  Theorem \ref{invpar} and Clifford's theorem $\Hh_{par}(G,R)$ is a strong semilattice of abelian groups $\Hh_{par}(G,R)=\bigcup\limits_{\xi\in \Lambda}\Hh_{par, \xi}(G,R),$ where $\Lambda$ is a semillatice isomorphic to the idempotents of $\Hh_{par}(G,R)$ and $\Hh_{par, \xi}(G,R)$ is a group for all $\xi\in \Lambda,$ (see \cite [Cor. IV.2.2]{HO}). Therefore it is useful to study the idempotents of $\Hh_{par}(G,R).$ 

We give 
the following.
\begin{prop}\label{idemp} Let $\alpha=\{\alpha_g \colon S_{g\m}\to S_g\}_{g\in G}$ be a unital  partial  action of $G$ on $S$ such that $S/R$ is a partial Galois extension. Then 
	\begin{enumerate}
		\item The family  $\hat\alpha=\left\{\hat\alpha_{(l,t)} \colon \left(\prod\limits_{g\in G}S_g\right)_{\!\!\!(l,t)\m}\to\left (\prod\limits_{g\in G}S_g\right)_{(l,t)}\right\}_{(l,t)\in G\times G},$ where 
		$$\left(\prod\limits_{g\in G}S_g\right)_{(l,t)}=\prod\limits_{g\in G}S_gS_lS_{t\m g }=\prod\limits_{g\in G}S_g1_l1_{t\m g },$$ 
		and 
		\begin{equation}\label{achat}\hat\alpha_{(l,t)}[(d_g1_{l\m}1_{t g})_{g\in G}]=(\af_l(d_g1_{l\m})1_{ltg})_{g\in G},\end{equation}  with $d_g\in S_g$  for all $g\in G,$
		is a unital partial action of $G\times G$ on $\prod\limits_{g\in G}S_g.$
		\item The  partial action $\af\ot \af^*=\{(\af\ot \af^*)_{(l,t)} \colon S_{l\m}\ot S_t\to  S_{l}\ot S_{t\m}\}_{(l,t)\in G\times G},$ where $(\af\ot \af^*)_{(l,t)} =\af_l\ot \af_{t\m} ,$ for all $(l,t)\in G\times G$  is partially $(G\times G)$-isomorphic to $\hat\alpha.$
		\item Let $E(S,\alpha)=\left(\prod\limits_{g\in G}S_g\right)^{\!\!\!\delta G}.$ Then  \begin{equation}\label{idempe}E(S,\alpha)=\left\{(d_g)_{g\in G}\in \prod\limits_{g\in G}S_g: \af_l(d_g1_{l\m})1_g=d_g1_l1_{lg}, \,\,\, \forall l\in G\right\},\end{equation} and  $[(E(\alpha),\hat\alpha_{G\times G/\delta G})]$ is an idempotent in $\Hh_{par}(G,R),$

where
\begin{equation}\label{acfix}\hat\alpha_{(1,l)\delta G}((d_g)_{g\in G})=(d_g1_{lg})_{g\in G}+ \sum_{i=2}^m\prod\limits_{j=1}^{i-1}(1_g-1_{lh_{j}}1_{h_jg})_{g\in G}(d_g1_{h_i} 1_{h_ig}1_{l g})_{g\in G}.\end{equation}

	\end{enumerate}
\end{prop}
\proof 1) First of all notice that for all $(l,t)\in G\times G$ we have 
\begin{align*}
\alpha'_{(l,t)}[(d_g1_{l\m}1_{gt})_{g\in G}]&=(\af_l(d_g1_{l\m})1_{lg}1_{ltg})_{g\in G}
\stackrel{\lambda=lgt}=(\af_l(d_g1_{l\m})1_{\lambda}1_{\lambda t\m})_{\lambda\in G},
\end{align*}  and $\hat\alpha_{(l,t)} $ is a well defined isomorphism whose inverse is $\hat\alpha_{(l\m,t\m)}.$ Moreover it is not difficult to check  conditions (P1)-(P4)  in \S 2.1.

2) Let $\varphi:S\otimes S\to \prod\limits_{g\in G}S_g$ defined by $\varphi(x\ot y)=(x\af_g(y1_{g\m}))_{g\in G},$ then by \cite[Theorem 4.1, iv)]{DFP} $\varphi$ is a $R$-algebra isomomorphism. We check that $\varphi$ satisfies conditions (i)-(ii) of Definition \ref{pariso}.

\noindent (i) For $(l,t)\in G\times G,$ we have  
$$\varphi[S_l\otimes S_{t\m}]=\prod\limits_{g\in G}S_l\af_g(S_{t\m}1_{g\m})=\prod\limits_{g\in G}S_gS_lS_{gt\m}=\left (\prod\limits_{g\in G}S_g\right)_{(l,t)}.$$

\noindent (ii) Take $x,y\in S,$ then 
\begin{align*}
(\varphi \circ (\af\ot \af^*)_{(l,t)})(x1_{l\m}\ot y1_t)&=\varphi  [\af_l(x1_{l\m})\ot \af_{t\m}( y1_t)]\\
&=(\af_l(x1_{l\m})\af_{gt\m}( y1_{g\m t})1_g)_{g\in G}\\
&\stackrel{g=lt\lambda}=(\af_l(x1_{l\m})\af_{l\lambda}( y1_{l\m \lambda\m})1_{lt\lambda})_{\lambda\in G}\\
&=(\af_l(x\af_{\lambda}( y 1_{\lambda\m})1_{t\lambda}1_{l\m}))_{\lambda\in G},
\end{align*}
on the other hand
\begin{align*}
( \af'_{(l,t)}\circ\varphi )(x1_{l\m}\ot y1_t)&=\af'_{(l,t)}[(x1_{l\m} \af_g(y1_{g\m})1_{gt})_{g\in G}]\\
&=\af'_{(l,t)}[(x\af_g(y1_{g\m})1_{l\m} 1_{gt})_{g\in G}]\\
&=(\af_{l}(x\af_g(y1_{g\m})1_{l\m} 1_{gt}))_{g\in G},
\end{align*} 
and thus $\varphi \circ (\af\ot \af^*)_{(l,t)}=\af'_{(l,t)}\circ\varphi $ in $S_{l\m}\ot S_t$, for all $l,t\in G.$ We conclude that  $\af\ot \af^*$ and $\af'_{(l,t)}$ are   partially $(G\times G)$-isomorphic.

3)  To check \eqref{idempe} notice that $(d_g)_{g\in G}\in \left(\prod\limits_{g\in G}S_g\right)^{\!\delta G},$ if and only if, for all  $l\in G$ 
$$
(d_g1_{l}1_{l g})_{g\in G}= \hat\alpha_{(l,l\m)}[(d_g1_{l\m}1_{l\m g})_{g\in G}]= (\alpha_{l}(d_g1_{l\m})1_{g})_{g\in G}.$$
Now  since $(S\ot S,\af\ot \af^*)$ and $\left(\prod\limits_{g\in G}S_g, \hat\alpha\right)$ are 
 $G\times G$-isomorphic, then  by Proposition \ref{welldef} the partial actions $((S\ot S)^{\delta_G},(\af\ot \af^*)'_{G\times G/\delta G})$  and $(E(S,\alpha), \hat\alpha'_{G\times G/\delta G})$ are partially $G$-isomorphic, thus $[(E(S,\alpha),\hat\alpha'_{G\times G/\delta G})]$ is an idempotent in $\Hh_{par}(G,R).$\\

Finally to prove \eqref{acfix}  note that $\hat\alpha_{(1,l)\delta G}((d_g)_{g\in G})$ equals
\begin{align*}&\af_{(1,l)}(d_g1_{(1,l\m)})_{g\in G}+ \sum_{i=2}^m\prod\limits_{j=1}^{i-1}((1_g)_{g\in G}-1_{(h_{j}, lh\m_{j})})\af_{(h_i, lh\m_i)}(d_g1_{(h_i\m, l\m h_i)})_{g\in G}\stackrel{\eqref{achat}}=\\
%&(d_g1_{lg})_{g\in G}+ \sum_{i=2}^m\prod\limits_{j=1}^{i-1}((1_g)_{g\in G}-(1_g1_{h_{j}}1_{l\m h_jg})_{g\in G})(\af_{h_i}(d_g1_{h_i\m} )1_{l g})_{g\in G})=\\
&(d_g1_{lg})_{g\in G}+ \sum_{i=2}^m\prod\limits_{j=1}^{i-1}((1_g)_{g\in G}-(1_g1_{h_{j}}1_{l\m h_jg})_{g\in G})(\af_{h_i}(d_g1_{h_i\m} )1_{l g})_{g\in G}\stackrel{\eqref{idempe}}=\\
%&(d_g1_{lg})_{g\in G}+ \sum_{i=2}^m\prod\limits_{j=1}^{i-1}((1_g)_{g\in G}-(1_g1_{h_{j}}1_{l\m h_jg})_{g\in G})(\af_{h_i}(d_g1_{h_i\m} )1_{l g})_{g\in G})=\\
&(d_g1_{lg})_{g\in G}+ \sum_{i=2}^m\prod\limits_{j=1}^{i-1}((1_g)_{g\in G}-(1_{lh_{j}}1_{h_jg})_{g\in G})(d_g1_{h_i} 1_{h_ig}1_{l g})_{g\in G},
\end{align*}
as desired. \endproof
\begin{rem} Let $E(S,\alpha)_g$ the projection of $E(S,\alpha)$ onto the $g$th coordinate. Then $R=E(S,\alpha)_e$ and $R1_g\subseteq E(S,\alpha)_g$ for $g\neq e.$
\end{rem}
Let $\mathcal E$ be the meet semilattice of idempotents of $\Hh_{par}(G,R),$ by Proposition \ref{idemp}, the element  $E(\alpha)$ is such that its equivalence class belongs to  $\mathcal E.$ From this we have that $\Hh_{par}(G,R)=\bigcup\limits_{E(S,\alpha)}\Hh_{par, E(S,\alpha)}(G,R),$ where $\Hh_{par, E(S,\alpha)}(G,R)$ is the subgroup of $\Hh_{par}(G,R)$ whise identity element is the class $[E(S,\alpha),\hat\alpha_{G\times G/\delta G})].$ Then the following is clear.
\begin{equation*}\label{Ealpha}
\Hh_{par, E(\alpha)}(G,R)=\{[S', \alpha']:[E(S',\alpha'),\hat\alpha'_{G\times G/\delta G})]=[E(S,\alpha),\hat\alpha_{G\times G/\delta G})]\}.
\end{equation*}
%where $\hat\alpha_{G\times G/\delta G}$  and $\hat\theta_{G\times G/\delta G}$ can be described using  equations  \eqref{1primagg} and \eqref{alphapprima} below.

Using \eqref{acfix} we get the following.
\begin{prop} Let $R\subseteq S'$ be a partial Galois extension with partial action $\alpha',$ and  $S'_g=S1'_{g},$ for all $g\in G$. Then $[S', \alpha']\in \Hh_{par, E(S,\alpha)}(G,R),$ if and only if,  there is a $R$-algebra homomorphism $f: E(S,\alpha)\to E(S', \alpha')$ such that $f( (1_g)_{g\in G})=(1'_g)_{g\in G}.$ 
\end{prop}
\begin{proof}  The part ($\Rightarrow$) is clear. Conversely, suppose that there is a  $f: E(S,\alpha)\to E(S', \alpha')$ such that $f( (1_g)_{g\in G})=(1'_g)_{g\in G},$ then by \eqref{1gh} we see that $f(E(S,\alpha)1_{(g,1)\delta G})\subseteq E(S',\alpha'),$ for all $g\in G,$ moreover by \eqref{acfix} we get  that $f$ satisfies ii) of Definition \ref{pariso}. Then the result follows from  Proposition \ref{pro9}.
\end{proof}
%
%\begin{cor}(INCOMPLETO) Let $R\subseteq S'$ be a partial Galois extension with partial action $\alpha'$ and  $S'_g=S1'_{g},$ for all $g\in G$. Then $[S', \alpha']=[E(\alpha), \hat\alpha_{G\times G/\delta G}],$  if and only if,  there is a $R$-algebra map $f:S'\to $ such that  for all $g\in G.$ 
%\end{cor}
%\begin{proof} ($\Rightarrow$) Now  define
%$f^* : S'\to E(\alpha)$ by $f^*(s')=(f(\sigma'_{g\m}(s'1_g))1_g)_{g\in G}$
%\end{proof}

\subsection{Examples and remarks}

The study of partial Galois extension of finite abelian groups reduces to the study  of partial Galois extension of cyclic groups.  Indeed let $G$ be a finite abelian group, then there are cyclic groups $G_1, G_2, \cdots, G_n $ such that $G=G_1\times G_2\times \cdots \times G_n.$ Further, for each $1\leq i\leq n,$ let $S_i/R$ be a partial Galois extension with group $G_i$ and partial action $\af_i.$ Consider $S=S_1\otimes_R S_2\otimes_R \cdots \otimes_R S_n,$ then by \cite[Proposition 2.9]{DPP} we have that $S/R$ is a partial Galois extension with group $G$ and partial action $\af=\af_1\otimes \af_2\otimes\cdots \otimes \af_n$. Thus, we have a map
$$\prod_{i=i}^n\Hh_{par}(G_i,R) \ni ([S_1, \alpha_1],[S_2, \alpha_2], \cdots, [S_n, \alpha_n])  \stackrel{\phi}\mapsto [S, \alpha]\in \Hh_{par}(G,R) .$$
To construct its inverse, let $S/R$ is a partial Galois extension of $G$ on $S$ with partial action $\alpha,$  for each $1\leq i\leq n,$ consider $H_i=G_1\times G_2\times \cdots G_{i-1}\times\{e\}\times G_{i+1} \times\cdots \times  G_n,$ then $H_i$ acts partially on $S.$ Write $S_i=S^{\alpha_{H_i}},$ then there is a group isomorphism $G/H_i\simeq G_i$ and by Proposition 3.4 $S_i/R$ is a partial Galois extension with group $G_i,$ and we  have a map
$$ \Hh_{par}(G,R) \ni [S, \alpha]  \stackrel{\varphi}\mapsto ([S_1, \alpha_1],[S_2, \alpha_2], \cdots, [S_n, \alpha_n])  \in \prod_{i=i}^n\Hh_{par}(G_i,R) .$$
It is not difficult to check that the map $\phi$ is a bijection with inverse $\varphi.$

\begin{rem}\label{rempartfix}
	Note that for $\sum_ix_i\ot y_i\in(A\ot B)^{\delta G}$ we have
	\begin{equation}\label{parfix}
	\alpha_g\ot \alpha_{g\m}(\sum_ix_i1_{g^{-1}}\ot y_i1_g)=\sum_ix_i1_g\ot y_i1_{g^{-1}}.
	\end{equation}
	Hence,
	\begin{align*}
	(\alpha_{g\m}\ot \alpha_1)(\sum_ix_i1_g\ot y_i1_{g^{-1}})&\stackrel{\eqref{parfix}}=(\alpha_{g\m}\ot \alpha_1)(\alpha_g\ot \alpha_{g\m})(\sum_ix_i1_{g^{-1}}\ot y_i1_g)\\
	&=\sum x_i1_{g\m}\ot \alpha_{g\m}(y_i1_g)\\
	&=(\alpha_1\ot \alpha_{g^{-1}})(\sum_ix_i1_{g^{-1}}\ot y_i1_g).
	\end{align*}
\end{rem}

Now we give some examples to ilustrate our results.

From now on $G$ will denote the cyclic group $G=\langle g\mid g^4=1\rangle$   of order 4. 
\begin{exe}\label{ex1}
	Let  $R$ be  a commutative ring and set $S=Re_1\oplus Re_2 \oplus Re_3,$ where $e_1, e_2, e_3$ are non-zero orthogonal idempotents whose sum is one and let  and the partial action of $G$ on $S$ given by \cite[Example 6.1]{DFP}.  That is $S_g=Re_1\oplus Re_2, S_{g^2}=Re_1\oplus Re_3, S_{g^3}=Re_2\oplus Re_3$, and\\\\
	$\alpha_{g}: S_{g^3} \to S_{g};\,  \alpha_{g}(e_2)=e_1, \, \alpha_{g}(e_3)=e_2,$\\\\
	$\alpha_{g^2}: S_{g^2}\to S_{g^2}; \, \alpha_{g^2}(e_1)=e_3, \, \alpha_{g^2}(e_3)=e_1$ and \\\\
	$\alpha_{g^3}: S_{g} \to S_{g^3};\,  \alpha_{g^3}(e_1)=e_2, \, \alpha_{g^3}(e_2)=e_3.$\\\\
	Hence, $S$ is an $\alpha$-partial  Galois extension of $R$. Let $H=\{1,g^2\}\leq G$ and $\mathcal{T}=\{1,g\}$ a transversal of $H$ in $G$. Then $S^{\alpha_H}=R(e_1+e_3)\oplus Re_2$ and the family $\alpha'_{G/H}=(D'_g, \af'_g)_{g\in \mathcal{T}}$ is a ${G/H}$-partial Galois extension of $S^{\alpha}=R,$ where by \eqref{1primagg} we have $\tilde{1}_H=1_S=\tilde{1}_{gH}$ 
	%$1'_{g}=1_{g} + (1_S-1_{g})1_{g^3}=e_1+e_2 +(e_3)(e_2+e_3)=1_S $,
 and by equations  \eqref{alfaprima}, \eqref{alphapprima}   we have 
	$$\tilde D_H=\tilde D_{gH}=[R(e_1+e_3)\oplus Re_2](e_1+e_2)=S^{\alpha_H}$$  $\alpha_H={\rm id}_{S^{\alpha_H}}$ and 
	$\alpha_{gH}(x)=\af_g(x1_{g^3})+\af_{g^3}(x1_g)(1_S-1_g),$ for all $x\in S^{\af_H}.$
\end{exe}

It follows by Example \ref{ex1} that in general if $\alpha_{G/H}$ is global, then the partial action $\alpha$ is not necessarilly global.

%\begin{rem}\label{gloobal} We notice that If $\alpha=(D_g,\af_g)_{g\in G}$ is a unital partial action of a cyclic group   $G$ generated by $\sigma$ on a ring $S$ and $D_\sigma=S,$ then $\alpha$ is global,  since $\alpha$ is unital we have that $D_{\sigma\m}=R$ and  $g\in G$ we have by (ii) of Definition \ref{pa} that $D_{\sigma g}=\af_\sigma(D_g),$ and the result follows.
%\end{rem}

\begin{exe}\label{ex3}
	
	Consider the  ring $S'=Re'_1\oplus Re'_2 $, where $e'_1, e'_2$ are non-zero orthogonal idempotents whose sum is one. We define a partial action $\theta$ of $G$ on $S'$ by taking $S'_1=S'$, $S'_{g}=Re'_2$, $S'_{g^2}=\{0\}$, $S'_{g^3}=Re'_1$, and setting  $\theta_1=id_{S'}$,
	$$\theta_{g}: S'_{g^3} \to S'_{g};\,  \theta_{g}(e'_1)=e'_2, \,\,\,\text{   and    }\,\,\,
	\theta_{g^3}: S'_{g} \to S'_{g^3};\,  \theta_{g^3}(e'_2)=e'_1.$$
	Note that, $S'^{\theta}=R$ and $\{x_1=y_1=e'_1, x_2=y_2=e'_2\}$ are  partial Galois coordinates of $S'$ over $R.$  We calculate the product $[(S',\theta^*)]\ast_{par}[(S',\theta)]$

	First we  get that 
% $(S\otimes S')^{\delta G},$ where 
	%$\delta G=\{(1,1), (g,g^3), (g^2,g^2), (g^3,g)\}.$
	%Indeed, for $x=r_1e'_1+r_2e'_2 \in S$ and $y=s_1e'_1+s_2e'_2 \in S',$ with $x\ot y\in (S\otimes S')^{\delta G},$ then 
	
	%$$\te_{g^3}\otimes \theta_{g^3}((x\otimes y)(e'_2\otimes e'_2))=(x\otimes y)(e'_1\otimes e'_1),$$ which implies 
	%\begin{multline*}\te_{g^3}\otimes \theta_{g^3}([(r_1e'_1+r_2e'_2)\otimes (s_1e'_1+s_2e'_2)](e_2\otimes e'_2))\\
	%=([(r_1e_1+r_2e_2)\otimes (s_1e'_1+s_2e'_2)](e_1\otimes e'_1))
	%\end{multline*}
%	$
%	\alpha_g\otimes \theta_{g^3}(r_2s_2e_2\otimes e'_2)=r_1s_1e_1\otimes e'_1,$ that is 
%	$r_2s_2e_1\otimes e'_1=r_1s_1e_1\otimes e'_1,$ and we get 
%	%\end{align*}
%	$$r_2s_2=r_1s_1.$$
%	Now, condition
%	$\te_{g}\otimes \theta_{g}((x\otimes y)(e'_1\otimes e'_1))=(x\otimes y)(e'_2\otimes e'_2)$
%	gives the same equality, $r_2s_2=r_1s_1$. Thus
	$$(S\otimes S')^{\delta G}=R(e'_1\otimes e'_1+e'_2\otimes e'_2)\oplus R(e'_1\otimes e'_2)\oplus R(e'_2\otimes e'_1).$$
	Now we use equation \eqref{1gh} to find the idempotents $1_{(x,1)\delta_G},$  with $x\in  G$. Let  $h_i=g^{i-1}, 1\leq i\leq 4.$ Then 
	
	\begin{align*}1_{(x,1) \delta G}&
	=1^*_x\otimes 1_S+\sum_{i=2}^4\prod\limits_{j=1}^{i-1}(1_S\otimes 1_{S}-1^*_{xh_{j}}\otimes 1_{h^{-1}_{j}})(1^*_{xh_i}\otimes 1_{h^{-1}_i})
	\\&=1_{x\m}\otimes 1_S+\sum_{i=2}^4\prod\limits_{j=1}^{i-1}(1_S\otimes 1_{S}-1_{(xh_{j})\m}\otimes 1_{h^{-1}_{j}})(1_{(xh_i)\m}\otimes 1_{h^{-1}_i})
	\end{align*}
	
%	But
%	
%	$$1_{x\m}\otimes 1_S= \left\{ \begin{array}{lcc}
%	e'_1\ot 1_S &  if   & x= g\\
%	0  & if   & x= g^2\\
%	e'_2\ot 1_S  &  if   & x= g^3\\
%	\end{array}
%	\right.$$
%	
%	Now we calculate the values $1_{(xh_{j})\m}\otimes 1_{h^{-1}_{j}},$ where $x\in G$ and $1\leq j\leq 4.$
%	
%	For $x=g$,
%	$$1_{g^3 h_{j}\m}\otimes 1_{h_{j}\m}= \left\{ \begin{array}{lcc}
%	e'_1\ot 1_S&  if   & j= 1\\
%	0&  if   & j= 2\\
%	0  &  if   & j= 3\\
%	1_S\ot e'_2  &  if   & j= 4\\
%	\end{array}
%	\right.$$   
%	
%	From this 
%	
%	
%	$$1_S\ot 1_S-1_{g^3 h_{j}\m}\otimes 1_{h_{j}\m}= \left\{ \begin{array}{lcc}
%	e'_2\ot 1_S&  if   & j= 1\\
%	1_S\ot 1_S&  if   & j= 2\\
%	1_S\ot 1_S  &  if   & j= 3\\
%	%1_S\ot e'_1  &  if   & j= 4\\
%	\end{array}
%	\right.$$   
%	
%	
%	Now we determine $\displaystyle\sum_{i=2}^4u_i,$ we have $u_2=u_3=0,$ and $u_4=e'_2\ot e'_2.$ 
 Then
\begin{itemize}
\item	$\tilde 1_{(g,1)}\delta G=e'_1\ot 1_S+e'_2\ot e'_2;$
\item	$\tilde 1_{(g^2,1)}\delta G=e'_2\ot e'_1+e'_1\ot e'_2;$
\item	$\tilde 1_{(g^3,1)}\delta G=e'_2\ot 1_S+e'_1\ot e'_1.$
\end{itemize}

%	
%	For $x=g^2$,
%	$$1_{g^2 h_{j}\m}\otimes 1_{h_{j}\m}= \left\{ \begin{array}{lcc}
%	0&  if   & j= 1\\
%	e'_2\ot e'_1&  if   & j= 2\\
%	0 &  if   & j= 3\\
%	e'_1\ot e'_2 &  if   & j= 4\\
%	\end{array}
%	\right.$$
%	From this
%	
%	
%	$$1_S\ot 1_S-1_{g^2 h_{j}\m}\otimes 1_{h_{j}\m}= \left\{ \begin{array}{lcc}
%	1_S\ot 1_S&  if   & j= 1\\
%	e'_1\ot e'_1 + e'_2\ot e'_2+e'_1\ot e'_2&  if   & j= 2\\
%	1_S\ot 1_S &  if   & j= 3\\
%	% e'_1\ot e'_1 + e'_2\ot e'_2+ e'_2\ot e'_1&  if   & j= 4\\
%	\end{array}
%	\right.$$
%	Then $u_2=e'_2\ot e'_1,$ $u_3=0,$  $u_4=(e'_1\ot e'_1 + e'_2\ot e'_2+e'_1\ot e'_2)(e'_1\ot e'_2)=e'_1\ot e'_2,$ and
	%$$1_{(g^2,1)}=e'_2\ot e'_1+e'_1\ot e'_2$$ 
%	For $x=g^3$,
%	$$1_{g h_{j}\m}\otimes 1_{h_{j}\m}= \left\{ \begin{array}{lcc}
%	e'_2\ot 1_S&  if   & j= 1\\
%	1_S\ot e'_1&  if   & j= 2\\
%	0&  if   & j= 3\\
%	0&  if   & j= 4\\
%	\end{array}
%	\right.$$
%	
%	
%	
%	
%	From this 
%	$$1_S\ot 1_S-1_{g h_{j}\m}\otimes 1_{h_{j}\m}= \left\{ \begin{array}{lcc}
%	e'_1\ot 1_S&  if   & j= 1\\
%	1_S\ot e'_2&  if   & j= 2\\
%	1_S\ot 1_S&  if   & j= 3\\
%	%1_S\ot 1_S&  if   & j= 4\\
%	\end{array}
%	\right.$$
%	Now  $u_2=e'_1\ot e'_1, $ $u_3=u_4=0,$ then 
	$$1'_{(g^3,1)}=e'_2\ot 1_S+e'_1\ot e'_1.$$
	
	From this we get,
	\begin{align*}
	\tilde D_{(g,1)\delta G}&=(S\otimes S')^{\delta G}1_{(g,1)\delta G}\\
	&=[R(e'_1\otimes e'_1+e'_2\otimes e'_2)\oplus R(e'_1\otimes e'_2)\oplus R(e'_2\otimes e'_1)]
	(e'_1\ot 1_S+e'_2\ot e'_2)\\
	&=R(e'_1\otimes e'_1+e'_2\otimes e'_2)\oplus R(e'_1\otimes e'_2),
	\end{align*}
	\begin{align*}
	\tilde D_{(g^3,1)\delta G}&=(S\otimes S')^{\delta G}1_{(g^3,1)\delta G}\\
	&=[R(e'_1\otimes e'_1+e'_2\otimes e'_2)\oplus R(e'_1\otimes e'_2)\oplus R(e_2\otimes e'_1)](e'_2\ot 1_S+e'_1\ot e'_1)
	\\&=R(e'_1\otimes e'_1+e'_2\otimes e'_2)\oplus R(e'_2\otimes e'_1)
	\end{align*}
	and
	\begin{align*}
	\tilde D_{(g^2,1)\delta G}&=(S\otimes S')^{\delta G}1_{(g^2,1) \delta G}\\
	&=[R(e'_1\otimes e'_1+e'_2\otimes e'_2)\oplus R(e_1\otimes e'_2)\oplus R(e'_2\otimes e'_1)](e'_1\ot e'_2+e'_2\ot e'_1)\\
	&=R(e'_1\otimes e'_2)\oplus R(e'_2\ot e'_1).
	\end{align*}
	Now, using equation (\ref{alphapprima}) we find $\hat\lambda _{(l,1)\delta G}$ with $l\in G,$ where $\lambda=\theta\ot \theta^*$.  We get
\begin{itemize}
\item 	$\hat\lambda _{(g,1)\delta G}(r(e'_1\otimes e'_1+e'_2\otimes e'_2)\oplus s(e'_2\otimes e'_1))=r(e'_1\otimes e'_1+e'_2\otimes e'_2)+s(e'_1\ot e'_12.$
\item $\hat\lambda _{(g^2,1)\delta G}(r(e'_1\ot e'_2)+s(e'_2\ot e'_1))=s(e'_1\ot e'_2)+r(e'_2\ot e'_1).$
\item $\hat\lambda _{(g^3,1)\delta G}=\hat\lambda\m _{(g,1)\delta G}.$
\end{itemize}

\end{exe}

\end{document}